\documentclass[11pt,reqno]{amsart}

\usepackage{amsmath,amssymb,amsthm,amsfonts}   
\usepackage{bbm} 
\usepackage{graphicx}   
\usepackage[colorlinks=true, linkcolor=blue, citecolor=magenta]{hyperref} 
\usepackage{url}
\usepackage{enumerate}
\usepackage[shortlabels]{enumitem}
\usepackage{caption,subcaption} 

\usepackage{tabu}
\usepackage[commentmarkup=footnote]{changes}

\usepackage{mathtools}      
\mathtoolsset{showonlyrefs} 

\usepackage{mathrsfs}
\usepackage{multicol}   
\usepackage{fullpage}   

\usepackage[foot]{amsaddr}

\usepackage{dsfont} 

\usepackage[section]{placeins}

\definechangesauthor[name=Daniel, color=red]{de}
\definechangesauthor[name=Felipe, color=blue]{fp}

\newcommand{\R}{\mathbb{R}}
\newcommand{\Z}{\mathbb{Z}}
\newcommand{\N}{\mathbb{N}}

\newcommand{\vc}[1]{\boldsymbol{#1}}
\newcommand{\ind}{\mathds{1}} 	

\DeclarePairedDelimiter\abs{\lvert}{\rvert}
\DeclarePairedDelimiter\norm{\lVert}{\rVert}

\DeclarePairedDelimiter\floor{\lfloor}{\rfloor}

\DeclareMathOperator{\supp}{supp}
\DeclareMathOperator{\BigO}{\mathcal{O}}
\let\dim\relax
\DeclareMathOperator{\dim}{dim_{\mathcal{H}}}

\DeclareFontFamily{U}{mathx}{\hyphenchar\font45}
\DeclareFontShape{U}{mathx}{m}{n}{
      <5> <6> <7> <8> <9> <10>
      <10.95> <12> <14.4> <17.28> <20.74> <24.88>
      mathx10
      }{}
\DeclareSymbolFont{mathx}{U}{mathx}{m}{n}
\DeclareFontSubstitution{U}{mathx}{m}{n}
\DeclareMathAccent{\widecheck}{0}{mathx}{"71}
\DeclareMathAccent{\wideparen}{0}{mathx}{"75}

\newtheorem{thm}{Theorem}[section]
\newtheorem{cor}[thm]{Corollary}
\newtheorem{prop}[thm]{Proposition}
\newtheorem{lem}[thm]{Lemma}

\newtheorem{rmk}[thm]{Remark}
\newtheorem*{rmk*}{Remark}
\newtheorem*{acknowledgements*}{Acknowledgements}

\theoremstyle{definition}

\definecolor{darkGreen}{RGB}{14,146,0}

\makeatletter
\@namedef{subjclassname@2020}{\textup{2020} Mathematics Subject Classification}
\makeatother

\title{Counterexamples for the fractal Schr\"odinger convergence problem 
with an Intermediate Space Trick}

\author[D. Eceizabarrena]{Daniel Eceizabarrena}
\address[D. Eceizabarrena]{Department of Mathematics and Statistics,
		University of Massachusetts Amherst,
		Amherst MA 01003, United States}
\email{eceizabarrena@math.umass.edu}

\author[F. Ponce-Vanegas]{Felipe Ponce-Vanegas}
\address[F. Ponce-Vanegas]{BCAM -- Basque Center for Applied Mathematics,
		Mazarredo 14, E48009 Bilbao, Basque Country -- Spain}
\email{fponce@bcamath.org}


\date{\today}

\begin{document}

\begin{abstract}
We construct counterexamples 
for the fractal Schr\"odinger convergence problem 
by combining
a fractal extension of Bourgain's counterexample and
the intermediate space trick of Du--Kim--Wang--Zhang. 
We confirm that the same 
regularity as Du's
counterexamples for weighted $L^2$ restriction
estimates is achieved for the convergence problem.
To do so, we need to construct the set of divergence explicitly 
and compute its Hausdorff dimension, 
for which we use the Mass Transference Principle, 
a technique originated from Diophantine approximation. 
\end{abstract}

\maketitle



\section{Introduction}

We study the convergence problem of the solutions of the 
Schrödinger equation to the initial datum in its fractal version. That is, if 
$u = e^{it\Delta}f$ is the solution to
\begin{equation}
\left\{
\begin{array}{l}
u_t = -\frac{ i}{2\pi} \, \Delta u, \\
u(x,0) = f(x),
\end{array}
\right.
\end{equation}
with $f \in H^s(\R^n)$, we look for the minimal Sobolev regularity $s$
so that 
\begin{equation}
\lim_{t \to 0} e^{it\Delta} f(x) = f(x) \quad \text{ for } \mathcal H^\alpha \text{-almost all } x \in \mathbb R^n, \quad \forall f \in H^s(\mathbb R^n),
\end{equation}
where $0 \leq \alpha \leq n$ and $\mathcal H^\alpha$ is the $\alpha$-Hausdorff measure.
In other words, we look for the exponent
\begin{equation}
s_c(\alpha) = 
	\inf\left\{s \ge 0 \, \, \mid \, \,  \lim_{t \to 0} e^{it\Delta}f = f\enspace 
		\mathcal{H}^\alpha\textrm{-a.e., } \quad \forall  f\in H^s(\R^n)\right\}.
\end{equation}
The case $\alpha = n$ for the Lebesgue measure is the original problem, 
proposed by Carleson in \cite{carleson1980}. 
The fractal refinement we here consider was studied later by
Sjögren and Sjölin \cite{sjogren_etal1989}, and by 
Barceló \textit{et al.} \cite{bercelo_etal2011}.

This problem, as well as variations of it, 
has received much attention over the past decades 
\cite{sjolin1987, vega1988, sjogren1989, moyua1996, lee2006, 
choLeeVargas2012, bourgain2013, choLee2014, wangZhang2019, lucaRogers2019,
pierce2020, eceizabarrena2020, Pierce2021, LiWangYan2021, compaan_etal2021, LiZhao2021}.
We discuss here with more detail the contributions to the fractal problem.

Concerning the Lebesgue case $\alpha = n$,
Carleson himself proved that $s_c(n) \le 1/4$ when $n = 1$.
This was confirmed to be optimal by Dahlberg and Kenig \cite{dahlberg_etal1982}, who provided a counterexample that implies $s_c(n) \ge 1/4$ in every dimension. 
After the  contribution of many authors, 
Bourgain's counterexample \cite{Bourgain2016}
and the positive results of Du, Guth and Li in $n=2$ \cite{du_etal2017},
and Du and Zhang in $n \geq 3$
\cite{du_etal2018} 
determined that the correct exponent is 
\begin{equation}
s_c(n) = \frac{n}{2(n+1)}.
\end{equation}

A preliminary result for the fractal case $\alpha < n$ is that of \v{Z}ubrini\'c \cite{zubrinic2002}, 
who showed that a function $f\in H^s(\R^n)$ with 
$s < (n - \alpha)/2$ need not be well-defined in a set of Hausdorff dimension $\alpha$.
In that case, since the initial datum itself is not well-defined,
we directly get
\begin{equation}\label{eq:Zubrinic}
s_c(\alpha) \ge (n - \alpha)/2, \qquad \text{ for all } \alpha \in [0,n].
\end{equation}

In the range $\alpha \le n/2$ the problem was solved by 
Barceló \textit{et al.} \cite[Proposition~3.1]{bercelo_etal2011}, who proved that
$s_c(\alpha) \le (n - \alpha)/2$ and thus showed that
\v{Z}ubrini\'c's bound \eqref{eq:Zubrinic} is best possible. 

Thus, we only need to focus on the case $\alpha > n/2$. 
In \cite[Theorem~2.3]{du_etal2018}, Du and Zhang proved that
\begin{equation}
s_c(\alpha) \le 
	\frac{n}{2(n + 1)} + \frac{n}{2(n + 1)}(n - \alpha), \qquad
\textrm{for } (n + 1)/2 \le \alpha \le n.
\end{equation}
The proof goes through the standard argument of using the maximal function, and
then this is reduced to the bound
\begin{equation} \label{eq:linearized_maximal_fun}
\norm{e^{it\Delta} f}_{L^2(w)} \le 
	C_\epsilon R^{\frac{\alpha}{2(n + 1)} + \epsilon}\norm{f}_{L^2(\R^n)}, \qquad
\textrm{for all } \epsilon > 0 \textrm{ and } R \ge 1,
\end{equation}
where $\supp\hat{f} \subset \{\xi \in \R^n \, : \,  \abs{\xi} \simeq 1\}$,
and $w \ge 0$ is a weight function that satisfies the following properties:
\begin{enumerate}[(i)]
\item $w$ is a sum of functions $\ind_Q$, where
$\{Q\}$ is a collection of unit cubes in a tiling of $\R^{n + 1}$;
\item $\supp w \subset B(0, R) = \{x \in \R^{n + 1} \mid \abs{x} \le R\}$;
\item $\int_{\R^n} w = R^\alpha$; 
\item $\int_{B_r(x)} w \le C_w r^\alpha$ for all $x \in \R^{n + 1}$ and $r > 0$.
\end{enumerate}

On the side of counterexamples, 
the best result we have so far is
\begin{equation} \label{eq:BestLowerBound}
s_c(\alpha) \ge 
	\frac{n}{2(n + 1)} + \frac{n - 1}{2(n + 1)}(n - \alpha), \qquad
\textrm{for } n/2 \le \alpha \le n.
\end{equation}
Luc\`a and Rogers \cite{luca_etal2019} proved this for 
$(3n + 1)/4 \le \alpha \le n$, for which they constructed counterexamples based on ergodic arguments, 
different from Bourgain's one in \cite{Bourgain2016} that is based on number theoretic arguments. 
Luc\`a and the second author adapted Bourgain's example 
to the fractal setting in \cite{LucaPonceVanegas2021} to prove \eqref{eq:BestLowerBound} in the whole range.

In this paper we construct further counterexamples 
that improve \eqref{eq:BestLowerBound}.
Defining   
\begin{gather}
s_{3,m}(\alpha) = \frac{n}{2(n-m+1)} + \frac{n-m-1}{2(n-m+1)}(n-\alpha), \\
s_{4,m}(\alpha) = \frac{n-m}{2(n-m+1)} + \frac{n-m}{2(n-m+1)}(n-\alpha), \\
s_{5,m}(\alpha) = \frac12 + \frac{n-m-2}{2(n-m-1)}(n-\alpha),
\end{gather}
and
\begin{equation}
\beta_{m} = n - (m-1) \,  \frac{n-m-1}{n-m-3},
\end{equation}
we prove the following theorem.
\begin{thm}\label{thm:Main_Theorem}
Let $m_0 = \floor{(n - 1)/3}$ and  $m_1 = \floor{n/2 - 1}$
and $0 \leq m \leq m_1$.
Then, 
\begin{itemize}
	\item When $n=2, 3$, then
	\begin{equation}
	s_c(\alpha) \geq s_{3,0}(\alpha), \qquad n/2 \leq \alpha \leq n.
	\end{equation}
	
	\item When $n=4,5,6,7$, then 
	\begin{equation}
	s_c(\alpha) \geq 
	\left\{ \begin{array}{ll}
	s_{3,m_0}(\alpha), & n/2 \leq \alpha \leq n - m_0, \\
	\max\{  s_{3,m-1}(\alpha), s_{4,m}(\alpha)   \}, & n - m \leq \alpha \leq n-m+1 \quad \text{ and } m=1, \ldots, m_0.
	\end{array}
	\right.
	\end{equation}
	
	\item When $n = 8, 9, 10, 11, 13$, then
	\begin{equation}
	s_c(\alpha) \geq 
	\left\{ \begin{array}{ll}
	s_{3,m_0+1}(\alpha)   , & n/2 \leq \alpha \leq \beta_{m_0+1}, \\
	\max\{  s_{3,m_0}(\alpha), s_{5,m_0+1}(\alpha)  \}, & \beta_{m_0+1} \leq \alpha \leq n-m_0,  \\
	\max\{  s_{3,m-1}(\alpha), s_{4,m}(\alpha)   \}, & n - m \leq \alpha \leq n-m+1 \quad \text{ and } m=1, \ldots, m_0.
	\end{array}
	\right.
	\end{equation}
	
	\item When $n \geq 12$ and $n \in 2\N$, then
	\begin{equation}
	s_c(\alpha) \geq 
	\left\{ \begin{array}{ll}
	\max\{  s_{3,m-1}(\alpha), s_{5,m}(\alpha)  \}, & \beta_{m} \leq \alpha \leq \beta_{m-1} \quad \text{ and } m=m_0+2, \ldots, m_1,  \\
	\max\{  s_{3,m_0}(\alpha), s_{5,m_0+1}(\alpha)  \}, & \beta_{m_0+1} \leq \alpha \leq n-m_0,  \\
	\max\{  s_{3,m-1}(\alpha), s_{4,m}(\alpha)   \}, & n - m \leq \alpha \leq n-m+1 \quad \text{ and } m=1, \ldots, m_0.
	\end{array}
	\right.
	\end{equation}
	
	\item When $n \geq 15$ and $n \in 2\N+1$, then
	\begin{equation}
	s_c(\alpha) \geq 
	\left\{ \begin{array}{ll}
	s_{3,m_1}(\alpha)   , & n/2 \leq \alpha \leq \beta_{m_1}, \\
	\max\{  s_{3,m-1}(\alpha), s_{5,m}(\alpha)  \}, & \beta_{m} \leq \alpha \leq \beta_{m-1} \quad \text{ and } m=m_0+2, \ldots, m_1,  \\
	\max\{  s_{3,m_0}(\alpha), s_{5,m_0+1}(\alpha)  \}, & \beta_{m_0+1} \leq \alpha \leq n-m_0,  \\
	\max\{  s_{3,m-1}(\alpha), s_{4,m}(\alpha)   \}, & n - m \leq \alpha \leq n-m+1 \quad \text{ and } m=1, \ldots, m_0.
	\end{array}
	\right.
	\end{equation}
	
\end{itemize}

\end{thm}

\begin{figure}[t]
\centering
\includegraphics[width=0.75\textwidth]{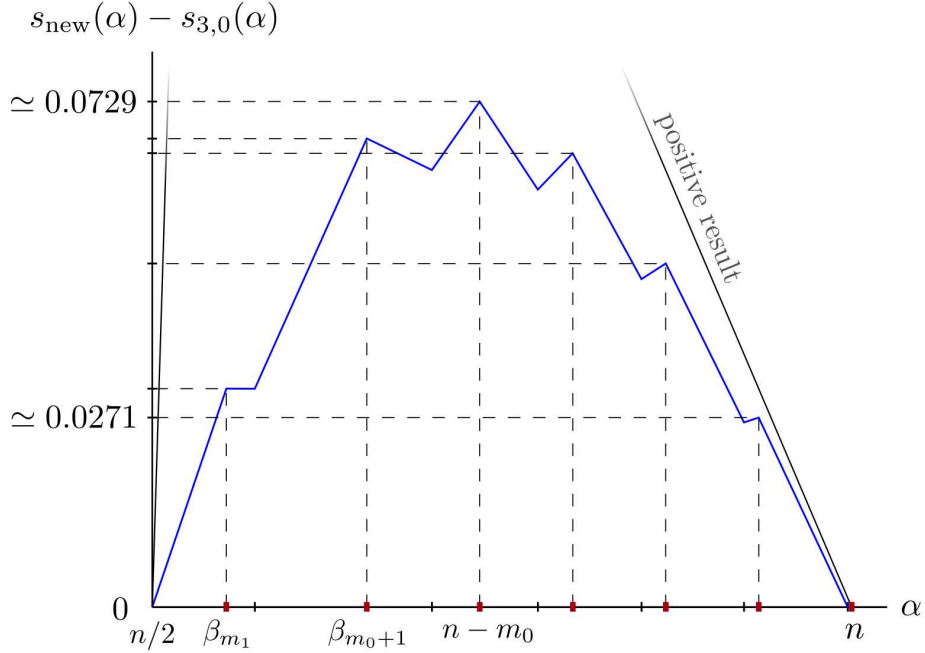}
\caption{Representation of Theorem~\ref{thm:Main_Theorem} for $n = 15$,
where we show the improvement with respect to the former lower bound \eqref{eq:BestLowerBound}. 
The positive result refers to Theorem~2.3 in \cite{du_etal2018}.}
\label{fig:regularity_n15}
\end{figure}

In the notation of Theorem~\ref{thm:Main_Theorem}, 
the best previous result in \cite{LucaPonceVanegas2021} is 
\begin{equation}
s_c(\alpha) \ge s_{3,0}(\alpha) \qquad
\textrm{for } n/2 \le \alpha \le n.
\end{equation}
See Figure~\ref{fig:regularity_n15} for a graphical comparison 
between the old and the new results.

The counterexamples combine 
the fractal extension of Bourgain's counterexample as presented in \cite{LucaPonceVanegas2021}, and
the intermediate space trick of Du--Kim--Wang--Zhang \cite{DKWZ20}.
In \cite{Du2020}, Du exploited this trick 
to construct counterexamples for \eqref{eq:linearized_maximal_fun}, which
are morally equivalent to counterexamples for convergence, 
except for one essential thing:
for convergence the weight $w$ must intersect every line $t \mapsto (x,t)$ 
in at most one interval of length $1$.
This additional restriction is evident in the fact that
Bourgain's counterexample needs Gauss sums, while
Du's examples do not.
The contribution of this paper thus is to confirm that
the numerology in Theorem~1.2 of \cite{Du2020} also
holds for convergence.

Unlike Du, we want to construct a fractal divergence set, which
demands further precautions.
As we did in \cite{eceizabarrena2021},
we compute the dimension of this set using 
the Mass Transference Principle proved in \cite{WangWu2021} (see also \cite{BeresnevichVelani2006}).

To compare Theorem~\ref{thm:Main_Theorem} with Du's Theorem~1.2 in \cite{Du2020},
the reader can use the relationship
\begin{equation}
s_{i,m}(\alpha) = \frac{n-\alpha + 1}{2}- \kappa_i(m + 1; \alpha, n+1),
\end{equation}
where $\kappa_i$ are functions defined by Du.
Notice that we chose our notation trying to make it easier
to compare our results with those of Du.  
The following dictionary might help:
\vspace*{2mm}
\begin{center}
{\tabulinesep=1.3mm
\begin{tabu}{ c c } \hline
\textbf{Du's Theorem~1.2} &
	\textbf{Theorem~\ref{thm:Main_Theorem}} \\
\hline\hline
$d$ & $n + 1$ \\
$j$ & $m + 1$ \\
$\kappa_3(j; \alpha, d) = \dfrac{d - j/2 - \alpha}{d -j + 1}$ &
	$s_{3, m}(\alpha)$ \\
$\kappa_4(j; \alpha, d) = \dfrac{d - \alpha}{2(d - j + 1)}$ & 
	$s_{4, m}(\alpha)$ \\
$\kappa_5(j; \alpha, d) = \dfrac{d - \alpha - 1}{2(d - j - 1)}$ &
	$s_{5, m}(\alpha)$ \\
\hline
\end{tabu}
}
\end{center}

\subsection*{Outline of the paper}

\begin{enumerate}

\item[\it Section~\ref{sec:Counterexample}:]
	For each integer $0 \le m \le n - 1$
	we construct a family of counterexamples,
	where $m$ is the dimension associated with
	the ``intermediate space trick''.
	We determine the set of divergence and
	the regularity of the initial data.

\item[\it Section~\ref{sec:Dimension}:]  
	We use the Mass Transference Principle to compute 
	the Hausdorff dimension of the set of divergence.
	
\item[\it Section~\ref{sec:Sobolev_Regularity}:]
	For each intermediate space dimension $m$ and
	the corresponding family of initial data,
	we fix a dimension $\alpha$ and 
	identify the data with maximum regularity.
	
\item[\it Section~\ref{sec:Maximum_Regularity}:]
	For a fixed dimension $\alpha$, 
	we determine the maximum regularity
	among data with different $m$.

\end{enumerate}

\subsection*{Notation}

\begin{itemize}
\item We denote $e(z) = e^{2\pi iz}$, and the Fourier transform of $f$ and the solution $e^{it\Delta}f$ are
\begin{equation}
\widehat{f}(\xi) = \int_{\mathbb R^n} f(x) \, e \big(-x\xi \big) \, dx
\qquad \text{ and } \qquad 
e^{it\Delta}f(x) = \int_{\mathbb R^n} \widehat f(\xi) \, e \big( x\xi + t|\xi|^2 \big)\, d\xi
\end{equation}

\item $B(a,r) = \{x : \abs{x-a}\le r\}$.

\item  $A\lesssim B$ means that $A\le CB$ for some constant $C>0$.
By $A\gtrsim B$ we denote the analog inequality.
We write $A\simeq B$ if $A\lesssim B$ and $B\lesssim A$.
When we want to stress some dependence of $C$ on a parameter $N$,
we write $A\lesssim_N B$. 

\item We write $c\ll 1$ as a shorthand of ``a sufficiently small constant''.
 
\item Size of sets: If $E \subset \mathbb R^n$ is a Lebesgue measurable set, 
then either $\abs{E}$ or $\mathcal H^n(E)$ denote its Lebesgue measure.
If $E$ is a finite set, then $\abs{E}$ is the number of elements. 

\item 
Given $0 \leq \alpha \leq n$ and $\delta >0$, 
the $(\alpha,\delta)$-Hausdorff content of $E \subset \R^n$ is 
\begin{equation}
\mathcal{H}_\delta^\alpha(E) 
	= \inf\bigg\{\sum_{j=1}^\infty (\textrm{diam}\, U_j)^\alpha \mid 
		E \subset \bigcup_{j=1}^\infty U_j \quad
		\textrm{such that} \quad 
		\textrm{diam}\, U_j < \delta \bigg\}, 
\end{equation}
and the $\alpha$-Hausdorff measure of $E$ is $  \mathcal{H}^\alpha(E) 
	= \lim_{\delta\to 0}\mathcal{H}^\alpha_\delta(E)$.
	The Hausdorff dimension of $E$ is 
	$\dim E
	= \inf\{\alpha \ge 0 \mid 
		 \mathcal{H}^\alpha(F) = 0\}$. 

\end{itemize}

\subsection*{Funding}
Daniel Eceizabarrena is supported by the Simons Foundation Collaboration Grant on Wave Turbulence (Nahmod's Award ID 651469), 
and by the National Science Foundation under Grant No. DMS-1929284 while he was in residence at ICERM - Institute for Computational and Experimental Research in Mathematics in Providence, RI, during the 
\textit{Hamiltonian Methods in Dispersive and Wave Evolution Equations} program.
Felipe Ponce-Vanegas is funded by the Basque Government through
the BERC 2018-2021 program; by the Spanish State Research Agency
through BCAM Severo Ochoa excellence accreditation SEV-2017-0718 and
the project PGC2018-094528-B-I00 - IHAIP; and by
a Juan de la Cierva -- Formation grant FJC2019-039804-I.

\subsection*{Acknowledgments}
We thank Renato Luc\`a for sharing with us all his insights on the problem.
We also thank Alex Barron; 
the idea for this project arose after an email exchange with him.


\section{Counterexample}
\label{sec:Counterexample}

Let $1 \leq m \leq n-1$, and split the variable $\xi \in \mathbb R^n$ as
\begin{equation}
 \xi = (\xi_1, \xi', \xi''), \qquad \text{ where }  \qquad (\xi_1, \xi', \xi'') \in \mathbb R \times  \mathbb R^{n-m-1} \times \mathbb R^m.
\end{equation}
Everywhere in this article, we use this notation for any variable in $\mathbb R^n$ or $\mathbb Z^n$.
Let $\phi \in \mathcal S(\mathbb R)$, $\varphi_1 \in \mathcal S(\mathbb R^{n - m - 1})$ and $\varphi_2 \in \mathcal S(\mathbb R^{m})$, 
all of which have positive Fourier transform with support 
in a ball $B(0,c)$, for $c \ll 1$.
Let also $\psi \in \mathcal S(\mathbb  R^{n-m-1})$ 
be a cutoff function supported in $B(0,c)$, for $c \ll 1$. 
Let $R \gg 1$ be the scale of the counterexample, 
which we should think of as tending to infinity, 
and $D_1, D_2 \gg 1$ be parameters,
which eventually will be appropriately chosen powers of $R$.

First, in Subsection~\ref{sec:Basic_Datum} we construct a
preliminary datum $f_R$ linked to a scale $R$.
Then, in Subsection~\ref{sec:Initial_Datum} we sum $f_R$ for dyadic $R$ 
to construct the counterexample for the convergence problem.

\subsection{A preliminary initial datum}\label{sec:Basic_Datum}

Let us first define the initial datum
\begin{equation}\label{eq:Initial_Datum_Fourier}
f(x) = f_R(x) = g(x_1) \, h_1(x') \, h_2(x'')
\end{equation}
such that
 \begin{equation}\label{eq:Initial_Datum_1}
\widehat g(\xi_1) = \widehat{\phi} \left( \frac{\xi_1 - R}{R^{1/2}} \right)
 \end{equation}
and
\begin{equation}\label{eq:Initial_Datum_Rest}
\widehat{h_1}(\xi') =  \sum_{\ell' \in \mathbb Z^{n-m-1} } \, \psi\left( \frac{\ell'}{R/D_1} \right) \, \widehat \varphi_1(\xi' - D_1 \, \ell' ), \qquad
\qquad
 \widehat{h_2}(\xi'') = \sum_{\substack{\ell'' \in \mathbb Z^m \\ |\ell''| \leq cR^{1/2}/D_2}} \, \widehat \varphi_2(\xi'' - D_2\, \ell'').
\end{equation}
Direct computation shows that 
\begin{equation}
\left\lVert g \right\rVert_2 \simeq R^{1/4}, \qquad \left\lVert h_1 \right\rVert_2 \simeq 
\left( \frac{R}{D_1} \right)^{(n-m-1)/2}, 
\qquad \left\lVert h_2 \right\rVert_2 \simeq \bigg( \frac{R^{1/2}}{D_2} \bigg)^{m/2},
\end{equation}
so 
\begin{equation}\label{eq:Norm_Of_Initial_Datum}
\left\lVert f_R \right\rVert_2 \simeq  R^{1/4} \, \left( \frac{R}{D_1} \right)^{(n-m-1)/2} \,  
\bigg( \frac{R^{1/2}}{D_2} \bigg)^{m/2}.
\end{equation}
Let us now study the evolution of this datum. 
We first do formal computations, which will be justified later. 

\begin{itemize}
	\item In the variable $x_1$,
\begin{equation}\label{eq:Evolution_Of_1_Pre}
\left| e^{it\Delta}g(x_1) \right| 
= 
R^{1/2} \, \left| \int_{\mathbb R} \widehat \phi(\xi_1)\, 
e \left(  \xi_1 R^{1/2}  ( x_1 + 2tR ) + t R |\xi_1|^2   \right) \, d\xi_1 \right|.
\end{equation}
If $|t| < 1/R$ and $ R^{1/2} |x_1 + 2Rt| < 1$, we get
\begin{equation}\label{eq:Evolution_Of_1}
\left| e^{it\Delta}g(x_1) \right| \simeq R^{1/2} \, \phi( R^{1/2} ( x_1 + 2R t ) ) \simeq R^{1/2}. 
\end{equation}

	\item For $x''$, we have
\begin{equation}\label{eq:Evolution_Of_2_Pre}
e^{it\Delta} h_2(x'')  
=
\sum_{\substack{\ell'' \in \mathbb Z^m \\ |\ell''| \leq cR^{1/2}/D_2}} \, 
e^{ 2\pi i \left( D_2 x'' \cdot \ell'' + tD_2^2\, |\ell''|^2  \right) } \,
 \int_{\mathbb R^m} \widehat \varphi_2(\xi'')\, 
e \big( \xi'' (x'' + 2tD_2 \ell'' ) +t|\xi''|^2  \big) \, d\xi''.
\end{equation}
The idea here is that if $|t| < 1/R$ and if we restrict the variable to $|x''| < 1$, 
all elements in the phase except $D_2 x'' \cdot \ell''$ are small. 
Thus, 
\begin{equation}
\left| e^{it\Delta} h_2(x'') \right| \simeq 
\Big|  \sum_{\substack{\ell'' \in \mathbb Z^m \\ |\ell''| \leq cR^{1/2}/D_2}} \, e^{ 2\pi i  \,  D_2 x'' \cdot \ell'' }   \Big| 
\end{equation}
If we choose $x'' = p'' / D_2 + \epsilon''$ for any $p'' \in \mathbb Z^m$ and $|\epsilon''| < R^{-1/2}$, 
then
\begin{equation}\label{eq:Evolution_Of_2}
\left| e^{it\Delta} h_2(x'') \right| \simeq \left( \frac{R^{1/2}}{D_2} \right)^m.
\end{equation}

	\item For $x'$, $h_1$ has a similar structure as $h_2$, so we obtain 
\begin{equation}\label{eq:Evolution_Of_3_Pre}
e^{it\Delta} h_1(x') 
=  \sum_{\ell' \in \mathbb Z^{n-m-1} } \, \psi\left( \frac{\ell'}{R/D_1} \right) \, 
e^{ 2\pi i \left( D_1 x' \cdot \ell' + tD_1^2 \, |\ell'|^2  \right)  } \, 
\int_{\mathbb R^{n-m-1}} \widehat \varphi_1(\xi')\, e\left( \xi' (x' + 2tD_1 \ell' )  + t|\xi'|^2  \right)  \, d\xi'.
\end{equation}
Again, restricting to $|x'| < 1$, the phase inside the integral is small, 
so we expect to have 
\begin{equation}
\left| e^{it\Delta} h_1(x')  \right| \simeq 
\Big|   \sum_{\substack{\ell' \in \mathbb Z^{n-m-1} \\ |\ell'| \leq cR/D_1}} \,    
e^{ 2\pi i \left( D_1 x' \cdot \ell' + tD_1^2 \, |\ell'|^2  \right)  }  \Big|.
\end{equation}
In this case we have a quadratic phase, 
so we take $x' = p' /(D_1 \, q) + \epsilon' $ and $t = p_1 / (D_1^2 q)$ 
such that $q \in 2 \mathbb N + 1$,  $p_1 \in \mathbb Z$ coprime with $q$, 
$p' \in \mathbb Z^{n-m-1}$ and $|\epsilon'| < R^{-1}$. 
That way, the exponential sum turns into the well-known Gauss sum, 
so we would obtain
\begin{equation}\label{eq:Evolution_Of_3}
\begin{split}
\left| e^{it\Delta} h_1(x')  \right| & 
\simeq 
\Big|   \sum_{\substack{\ell' \in \mathbb Z^{n-m-1} \\ |\ell'| \leq cR/D_1}} \, e\Big( \frac{p' \cdot \ell' +  p_1 \, |\ell'|^2}{q} \Big)   \Big| 
= \prod_{i=2}^{n-m} \Big|  \sum_{n =-cR/D_1}^{cR/D_1} \,  e \Big( \frac{p_i n +  p_1 \, n^2}{q}  \Big)     \Big|
\\ 
& \simeq \prod_{i=2}^{n-m} \frac{R}{qD_1} \, \sqrt{q} 
 = \left(  \frac{R}{D_1 \, q^{1/2}}\right)^{n-m-1}.
\end{split}
\end{equation}
	
\end{itemize}
Thus, combining \eqref{eq:Evolution_Of_1}, 
\eqref{eq:Evolution_Of_2} and \eqref{eq:Evolution_Of_3}
we expect to obtain
\begin{equation}\label{eq:Evolution}
\left| e^{it\Delta} f_R(x) \right|  \simeq R^{1/2} \, \left(  \frac{R}{D_1 \, q^{1/2}}\right)^{n-m-1} \, \bigg( \frac{R^{1/2}}{D_2} \bigg)^m,
\end{equation}
subject to the restrictions
\begin{equation}\label{eq:Restrictions}
t = \frac{p_1}{D_1^2 \, q}, \quad x_1 \in B^1 \left( \frac{R}{D_1^2} \frac{p_1}{q}, \frac{1}{R^{1/2}} \right), \quad  x' \in B^{n-m-1} \left( \frac{p'}{D_1\, q}, \frac{1}{R} \right), \quad x'' \in B^m \left( \frac{p''}{D_2}, \frac{1}{R^{1/2}} \right),
\end{equation}
where $q \in 2\, \mathbb N + 1$ and $ p \in \mathbb Z^n$  such that $\operatorname{gcd}(p_1,q) = 1$.
In view of this, let us define the slabs 
\begin{equation}\label{eq:Slabs}
E_R(p,q) 
= B^1 \left( \frac{R}{D_1^2} \frac{p_1}{q}, \frac{1}{R^{1/2}} \right) 
\times 
B^{n-m-1} \left( \frac{p'}{D_1\, q}, \frac{1}{R} \right)
\times
B^m \left( \frac{p''}{D_2}, \frac{1}{R^{1/2}} \right).
\end{equation}
All these formal computations,
 together with \eqref{eq:Norm_Of_Initial_Datum},
motivate the following proposition:
\begin{prop}\label{prop:Initial_Datum_Basic}
Let $1 \leq m \leq n-1$, $R \gg 1$ and $D_1,D_2,Q \gg 1$. 
Let $q \in 2\mathbb N + 1 $ be such that $Q/2 \leq q < Q$
and $p \in \mathbb Z^n$ such that $\operatorname{gcd}(p_1,q) = 1$. Then, 
letting $t = p_1 / (D_1^2 q)$, we have
\begin{equation}\label{eq:thm:Main_Lower_Bound}
\frac{ \left| e^{it\Delta} f_R(x) \right| }{\lVert f_R \rVert_2} \simeq R^{1/4} \, \left( \frac{R}{D_1 Q} \right)^{(n-m-1)/2} \, \left( \frac{R^{1/2}}{D_2} \right)^{m/2 }, \qquad \forall x \in E_R(p,q). 
\end{equation}
Moreover, if $1/10 \leq |x_1| \leq 1$, then the time satisfies $t \simeq 1/R$. 
\end{prop}

\begin{proof}
Let us first check that $t \simeq 1/R$. 
Indeed, from the definition of $E_R(p,q)$, we have 
$x_1 \in B(Rt, R^{-1/2})$, which implies 
\begin{equation}
 1/20 \leq  x_1 - R^{-1/2}  \leq Rt \leq x_1 + R^{-1/2} \leq 2,
\end{equation}
if $R$ is large enough.

The main estimate \eqref{eq:thm:Main_Lower_Bound} follows by combining 
\eqref{eq:Evolution_Of_1}, \eqref{eq:Evolution_Of_2} and \eqref{eq:Evolution_Of_3}
with \eqref{eq:Norm_Of_Initial_Datum}. 
Thus, it suffices to justify \eqref{eq:Evolution_Of_1}, \eqref{eq:Evolution_Of_2} and \eqref{eq:Evolution_Of_3}.

Estimate \eqref{eq:Evolution_Of_1}
follows from direct computation.
Indeed, from \eqref{eq:Evolution_Of_1_Pre} we write
\begin{align}
	\left| e^{it\Delta}g(x_1) \right| 
   & = R^{1/2} \, \left| \int_{\mathbb R} \widehat \phi(\xi_1)\, e\left( \xi_1 R^{1/2}  ( x_1 + 2tR ) + t R |\xi_1|^2  \right)  \, d\xi_1 \right| \\
	&\ge R^{1/2}\abs[\bigg]{\int_{\mathbb R} \widehat \phi(\xi_1)\, 
		\cos \big(2\pi \big(\xi_1 R^{1/2}  ( x_1 + 2tR ) + t R |\xi_1|^2 \big) \big) \, d\xi_1 }.
\end{align}
Asking $\abs{R^{1/2}  ( x_1 + 2tR )}< 1 $ and $ \abs{tR} < 1$,
since $\supp\widehat{\phi} \subset [-c, c]$ for $c$ small enough, 
we get
\begin{equation}
	\Big| \xi_1 R^{1/2}  ( x_1 + 2tR ) + t R |\xi_1|^2 \Big| < 1/10,
\end{equation} 
and therefore $|e^{it\Delta}g(x_1)| \gtrsim R^{1/2}$.

We prove \eqref{eq:Evolution_Of_2} similarly.
From \eqref{eq:Evolution_Of_2_Pre} we have
\begin{equation}
	e^{it\Delta} h_2(x'')  = 
		\int_{\mathbb R^m} \widehat \varphi_2(\xi'')\,
		\sum_{\substack{\ell'' \in \mathbb Z^m \\ |\ell''| \leq cR^{1/2}/D_2}} \, 
			e \Big(D_2 x'' \cdot \ell'' + tD_2^2\, |\ell''|^2 + 
				\xi'' (x'' + 2tD_2 \ell'' ) + t|\xi''|^2 \Big) \, 
		d\xi''.
	\end{equation}
Let $x'' = p''/D_2 + \epsilon''$ with $|\epsilon''| \leq R^{-1/2}$.
Since $\supp \widehat{\varphi}_2 \subset [-c, c]$, 
choosing $c$ small enough we get
\begin{equation}
	\Big| D_2 \epsilon'' \cdot \ell'' + tD_2^2\, |\ell''|^2 + 
	\xi'' (x'' + 2tD_2 \ell'' ) + t|\xi''|^2 \Big| < 1/10,
\end{equation}
and thus $ \abs{e^{it\Delta} h_2(x'')} \gtrsim (R^{1/2}/D_2)^m$.

Estimate \eqref{eq:Evolution_Of_3} is more technical. 
Let $x' = p'/(D_1 q) + \epsilon'$ with $|\epsilon'| < R^{-1}$, 
and write
\begin{equation}\label{eq:Evolution_1_Rigorous}
\abs[\big]{e^{it\Delta} h_1(x')} = 
	\abs[\Big]{ \sum_{\ell' \in \mathbb Z^{n-m-1} } \, \zeta(\ell')\, 
	e \Big(  \frac{p' \cdot \ell' + p_1 |\ell'|^2}{q}  \Big)    },
\end{equation} 
where
\begin{equation}\label{eq:Zeta_First}
\zeta(\ell') = \psi \left( \frac{\ell'}{R/D_1} \right) \, e^{2\pi i D_1 \epsilon' \cdot \ell'} \, 
\int \widehat \varphi_1(\xi')\, e \left( \xi' (x' + 2tD_1 \ell')  +  t|\xi'|^2   \right) \, d\xi'.
\end{equation}
To bound \eqref{eq:Evolution_1_Rigorous}, 
we use a simplified version of \cite[Lemma 3.4]{eceizabarrena2021}. 
\begin{lem}[Lemma 3.4 of \cite{eceizabarrena2021}]\label{thm:perturbation}
Let $d \in \mathbb N$ and $f(m) = a \, |m|^2 + b \cdot m$ such that $a \in \mathbb Z$ and $b \in \mathbb  Z^d$. 
Let also $\zeta\in C_0^\infty(\mathbb R^d)$ and define the discrete Laplacian $\widetilde\Delta$ by
\begin{equation}
\widetilde{\Delta} \zeta(y) = \sum_{j=1}^d \big( \zeta(y + e_j) + \zeta(y - e_j) - 2\zeta(y) \big), \qquad y \in \mathbb R^d,
\end{equation}
where $(e_j)_{j=1}^d$ is the canonical basis of $\mathbb R^d$. 
Assume that $\zeta$ is supported in $B(0,L)$ for some $L >0$, and moreover that 
$\norm{\widetilde\Delta^N\zeta}_\infty\lesssim_N L^{-2N}$ for every $N \in \mathbb N$.
Then,
\begin{equation} \label{eq:thm:perturbation}
\sum_{m\in \Z^d} \zeta(m) \,  e \Big( \frac{f(m)}{q} \Big) = 
	\bigg(\frac{1}{q^d}\sum_{m\in\Z^d} \zeta(m)\bigg)\sum_{l\in\Z_q^d} e \Big( \frac{f(l)}{q} \Big) + 
	\BigO_N\bigg(q^{d/2}\left(\frac{L}{q}\right)^{d-2N}\bigg)
\end{equation}
for any integer $N> d/2$.
\end{lem}

We use the lemma with $d = n-m-1$ and 
 $L = R/D_1$. 
 Rewrite $\zeta$ in \eqref{eq:Zeta_First} as
\begin{equation}\label{eq:Zeta}
\zeta(\ell') = 
	\psi \left( \frac{\ell'}{L} \right) \, e^{2\pi i \delta \,  \cdot \ell' / L} \, 
		\int \widehat \varphi_1(\xi')\, e\left( \xi' (x' + 2\tau \ell'/L) +  t|\xi'|^2  \right)  \, d\xi',
\end{equation}
where $\delta = R \epsilon'$ and $\tau = Rt$ satisfy $|\delta|, |\tau| < 1$. 
Notice that $\zeta$ is supported in $B(0, L)$.
On the other hand, we have
 $\lVert \widetilde{\Delta}^N \zeta  \rVert_\infty  \lesssim \sup_{|\alpha| = 2N} \lVert \partial^\alpha \zeta \rVert_\infty $,
where $\alpha = (\alpha_2, \ldots, \alpha_{n-m})$ denotes a multi-index.
Thus, it suffices to bound $\lVert \partial^\alpha \zeta \rVert_\infty$
uniformly in $x'$ and $t$. 
Write
\begin{equation}
\partial^\alpha \zeta(y) =  \int \widehat \varphi_1(\xi')\, e^{2\pi i  \left( \xi' \cdot x' + t|\xi'|^2 \right) } \, 
\partial^\alpha  \Big[  \psi \left( \frac{y}{L} \right) \, e^{2\pi i ( \delta +  2  \tau \xi' ) \cdot y/L } \Big] \, d\xi', \qquad y \in \mathbb R^{n-m-1}, \, \, |y| \leq L.
\end{equation}
Calling $A(z) = \psi(z)\, e^{2\pi i (\delta + 2\tau \xi') \cdot z}$, we have
\begin{equation}
 \partial^\alpha  \Big[  \psi \left( \frac{y}{L} \right) \, e^{2\pi i \,  ( \delta +  2  \tau \xi' ) \cdot y/L } \Big] = \frac{1}{L^{2N}} \, \partial^\alpha A(y/L),
\end{equation} 
and since $\partial^\alpha A(z)$ is uniformly bounded in $|\delta|, |\tau|, |z| < 1$,
we get $\lVert \partial^\alpha \zeta \rVert_\infty \lesssim_N L^{-2N}$. 
Thus, by Lemma~\ref{thm:perturbation}, 
we estimate \eqref{eq:Evolution_1_Rigorous} as
\begin{equation}\label{eq:Hard_Bound}
\begin{split}
\Big| e^{it\Delta} h_1(x') \Big| 
& = \frac{1}{q^{n-m-1}}\, \Big| \sum_{\ell' \in \mathbb Z^{n-m-1}} \zeta(\ell') \Big| \,   \Big|  \sum_{ \ell' \in \Z_q^{n-m-1}}  \, e \Big(   \frac{p' \cdot \ell' + p_1 |\ell'|^2}{q}  \Big)   \, \Big|  \\
& \qquad \qquad \qquad \qquad \qquad \qquad \qquad \qquad \qquad +  \mathcal O_N \left( q^{(n-m-1)/2}\,  \bigg(\frac{L}{q}\right)^{n-m - 1 - 2N}\bigg).
\end{split}
\end{equation}
Since the phase of $\zeta$ in \eqref{eq:Zeta} is small, 
by the same procedure we used for \eqref{eq:Evolution_Of_2} we get
\begin{equation}
\Big| \sum_{\ell' \in \mathbb Z^{n-m-1}} \zeta(\ell') \Big| \simeq L^{n-m-1}.
\end{equation} 
Also, since $\operatorname{gcd}(p_1,q) = 1$, $q$ is odd and $q \simeq Q$, the Gauss sums in \eqref{eq:Hard_Bound} satisfy
\begin{equation}
\Big|  \sum_{ \ell' \in \Z_q^{n-m-1}}  \, e \left( \frac{p' \cdot \ell' + p_1 |\ell'|^2}{q} \right)  \, \Big|  \simeq  Q^{(n-m-1)/2}.
\end{equation}
Thus, taking $N > (n-m-1)/2$ and replacing $L = R/D_1$, from \eqref{eq:Hard_Bound} we get
\begin{equation}
\begin{split}
\Big| e^{it\Delta} h_1(x') \Big| 
&  \simeq  Q^{(n-m-1)/2} \left( \frac{L}{Q} \right)^{n-m-1}  + \mathcal O_N \bigg( Q^{(n-m-1)/2}\,  \Big( \frac{L}{Q} \Big)^{ n - m - 1 - 2N}\bigg) \\
& \gtrsim  Q^{(n-m-1)/2} \left( \frac{L}{Q} \right)^{n-m-1} \\
& = \left( \frac{R}{D_1 Q^{1/2}} \right)^{n-m-1},
\end{split}
\end{equation}
which proves \eqref{eq:Evolution_Of_3}.
\end{proof}

Roughly speaking, Proposition~\ref{prop:Initial_Datum_Basic} would suffice 
in the case of the Lebesgue measure $\alpha = n$. 
Indeed,  given that $Q$, $D_1$ and $D_2$ will be certain powers of $R$, 
we will be able to find an exponent $s_m = s_m(Q, D_1, D_2)$ such that
\begin{equation}\label{eq:Heuristic_Exponent}
\frac{ \left| e^{it\Delta} f_R(x) \right| }{\lVert f_R \rVert_2} \simeq R^{1/4} \, \left( \frac{R}{D_1 Q} \right)^{(n-m-1)/2} \, \bigg( \frac{R^{1/2}}{D_2} \bigg)^{m/2 } = R^{s_m}, \qquad \forall x \in E_R(p,q). 
\end{equation}
Since the estimate does not depend on the particular choice of $p,q$
but rather on the size $q \simeq Q$, 
then \eqref{eq:Heuristic_Divergence} holds for $F_R = \bigcup_{q \simeq Q} \bigcup_{p} E_R(p,q)$.
Consequently, up to checking that $\mathcal H^n(F_R \cap B(0,1)) \simeq 1$ for all $R$, 
we would be able to write 
\begin{equation}\label{eq:Heuristic_Divergence}
\frac{ \left\lVert \sup_t \left| e^{it\Delta} f_R \right| \right\rVert_{L^2(B(0,1))}  }{\lVert f_R \rVert_{H^{s_m - \epsilon}}} \gtrsim R^{\epsilon}, \qquad \forall R \gg 1, 
\end{equation}
which would disprove the standard maximal estimate,
which is equivalent to the almost everywhere convergence property,
in $H^s(\R^n)$ for all $s < s_m$.  

However, in the fractal case $\alpha <n$, where we ask for
 almost everywhere convergence 
with respect to the $\mathcal H^\alpha$ measure, 
the maximal characterization does not work. 
This means that we need to construct a divergent counterexample explicitly.

\subsection{Construction of the counterexample}
\label{sec:Initial_Datum}

Let $\alpha < n$. 
To find a counterexample for the $\mathcal H^\alpha$ almost everywhere convergence property, 
we need to construct a function $f \in H^s(\R^n)$ whose set of divergence $F$ satisfies 
$\operatorname{dim}_{\mathcal H} F = \alpha$.
Moreover, we look for the biggest possible Sobolev regularity $s$. 

The standard way to do this is to sum dyadically
the data $f_R$ we constructed in the previous section. 
For every $j \in \mathbb N$, let $R_j = 2^j$.
As before, assume that $Q$, $D_1$ and $D_2$ are powers of $R$
so that $s_m = s_m(Q, D_1, D_2)$ is well-defined in \eqref{eq:Heuristic_Exponent}.
Define
\begin{equation}\label{eq:Initial_Datum}
f(x) = \sum_{j \geq K_0} \, j \, \frac{f_{R_j}(x)}{R_j^{s_m}\, \lVert f_{R_j} \rVert_2}
\end{equation}
for some $K_0$ large enough. 
Observe that $f \in H^s(\R^n)$ for every $s < s_m$ because  
\begin{equation}
\lVert f \rVert_{H^s} 
\leq \sum_{j \geq K_0}  j \frac{\lVert  f_{R_j} \rVert_{H^s}}{R_j^{s_m}\, \lVert f_{R_j} \rVert_2} 
\simeq \sum_{j \geq K_0}  \frac{j}{R_j^{s_m - s}   } < \infty.
\end{equation}
As suggested at the end of the previous subsection, 
since the estimate in Proposition~\ref{prop:Initial_Datum_Basic}
does not depend on $p,q$ but only on $Q$, 
we work with  
\begin{equation}\label{eq:F_R}
F_k = \bigcup_{\substack{ Q_k / 2 \leq q < Q_k \\ q \in 2\mathbb N + 1 }} \, \,  \bigcup_{ p \in G(q)  } E_{k}(p,q),
\end{equation}
where we denote $E_k(p,q) = E_{R_k} (p,q)$ and 
$G(q) = \left\{ p \in \mathbb Z^n \, : \,  \operatorname{gcd}(p_1,q) = 1 \right\}$. 
This way, by Proposition~\ref{prop:Initial_Datum_Basic} we have
\begin{equation}\label{eq:Big_Piece}
\frac{  \big| e^{it\Delta} f_{R_k} (x) \big|  }{ R_k^{s_m}\, \lVert f_{R_k} \rVert_2  }  \simeq 1, \qquad \forall x \in F_k, \qquad \forall k \in \mathbb N
\end{equation}
However, this only accounts for the behavior of the piece $f_{R_k}$. 
We show next that the contribution of 
the remaining $f_{R_j}$ with $ j\neq k$ is much smaller.

\begin{prop}\label{prop:Initial_Datum}
Let $K_0 \in \mathbb N$ be large enough and $k \geq K_0$.
Let $x \in F_k \cap B(0,1)$ be such that $1/10 < |x_1|\leq 1$. 
Then, there exists a time $t = t(x) \simeq R_k^{-1}$ such that
$
\abs[\big]{ e^{i t(x) \Delta} f(x)} \gtrsim k.
$
\end{prop}

With this proposition, 
the construction of the counterexample will be concluded 
if we can take the limit $k \to \infty$. 
For that, we need points that lie in infinitely many sets $F_k$.
The set of divergence is thus
\begin{equation}\label{eq:Divergence_Set}
F = \limsup_{k\to\infty} F_k  = \bigcap_{K \in \mathbb N}  \bigcup_{k \geq K} F_k.
\end{equation} 

\begin{cor}\label{cor:Divergence}
Let $F = \limsup_{k \to \infty} F_k$.
Then, 
\begin{equation}
\limsup_{t\to 0} \big| e^{it\Delta} f(x) \big|  = 
	\infty, \qquad \forall x \in F \cap A(1/10,1).
\end{equation}
\end{cor}
\begin{proof}[Proof of Corollary~\ref{cor:Divergence}]
If $x \in F$, then there exists a sequence $k_n$ such that $x \in F_{k_n}$ for all $n \in \mathbb N$. 
By Proposition~\ref{prop:Initial_Datum}, there exists 
a sequence of times $t_n = t_n(x)$ such that 
$t_n \simeq 1/R_{k_n}$ and $|e^{it_n \Delta}f(x)| \gtrsim k_n$ for all $n \in \mathbb N$. 
Thus, since $\lim_{n \to \infty} t_n(x) = 0$, we get 
\begin{equation}
\limsup_{t \to 0} \big| e^{it\Delta} f(x) \big| \geq \lim_{n \to \infty} \big| e^{it_n(x) \Delta} f(x) \big| = \infty.
\end{equation}
\end{proof}

In view of Corollary~\ref{cor:Divergence}, 
the main goal turns to computing the Hausdorff dimension of $F$. 
We do that in Section~\ref{sec:Dimension}. 
To conclude this section, we prove Proposition~\ref{prop:Initial_Datum}.

\begin{proof}[Proof of Proposition~\ref{prop:Initial_Datum}]
Fix $k \geq K_0$ and take $x \in F_k$.
According to \eqref{eq:Initial_Datum}, the solution looks like
\begin{equation}
\sum_{j \geq K_0} j \, \frac{e^{it\Delta}f_{R_j}(x)}{R_j^{s_m}\, \lVert f_{R_j} \rVert_2}.
\end{equation}
We first focus on the contribution of the piece $e^{it\Delta}f_{R_j}$ with $j=k$.
Since $x \in F_k$, there are $p_1$ and $q \simeq Q$ such that $x \in E_{R_k}(p,q)$, 
and thus, by Proposition~\ref{prop:Initial_Datum_Basic},  
there is a time $t(x) = p_1/(D_1 q)$ such that $t(x) \simeq 1/R_k$ and  
\begin{equation}\label{eq:Big_Piece_2}
\frac{  \big| e^{it(x)\Delta} f_{R_k} (x) \big|  }{ R_k^{s_m}\, \lVert f_{R_k} \rVert_2  }  \simeq 1.
\end{equation}

Now we want to measure the contribution of $e^{it(x)\Delta}f_j(x)$ for $j \neq k$.
We are going to prove 
\begin{equation}\label{eq:Small_Pieces}
 \frac{  \big| e^{it(x)\Delta} f_{R_j} (x) \big|  }{ R_j^{s_m}\, \lVert f_{R_j} \rVert_2  } \lesssim \frac{1}{j \, R_j}, \qquad \forall j \neq k.
\end{equation}
If this holds, then joining \eqref{eq:Big_Piece_2} and \eqref{eq:Small_Pieces}  we get
\begin{equation}
\big| e^{it(x)\Delta} f(x) \big| 
\gtrsim k - C\sum_{j \neq k} \frac{1}{R_j} 
\geq k - C\sum_{j =1}^\infty \frac{1}{2^j} 
\geq k/2, 
\qquad \text{for } K_0 \gg 1,
\end{equation}
which would conclude the proof.

To prove \eqref{eq:Small_Pieces}, the idea is that
the term $e^{it\Delta}g_{R_j}(x_1)$ in \eqref{eq:Evolution_Of_1}
localizes the solution
to the $n$-plane $T_j = \{x \, :  \,  |x_1+2R_jt| < R_j^{1/2}\}$. 
Thus, if $j \neq k$, the planes $T_j$ and $T_k$ are disjoint 
except in a neighborhood of the origin.
Consequently, 
if $|x_1| > 1/10$, the contribution of 
$e^{it\Delta}g_{R_j}(x_1)$ in the plane $T_k$ is very small.

Let us formalize the previous paragraph. First, we directly bound the contribution in the variables $x'$ and $x''$.
From \eqref{eq:Evolution_Of_2_Pre} and \eqref{eq:Evolution_Of_3}, we get
\begin{equation}
\big|  e^{it\Delta}h_{1, R_j}(x') \big| \lesssim \left( \frac{R_j}{D_1} \right)^{n-m-1} 
\qquad \text{ and } \qquad 
\big|  e^{it\Delta}h_{2, R_j}(x'') \big| \lesssim \bigg( \frac{R_j^{1/2}}{D_2} \bigg)^{m}, 
\end{equation}
and thus 
\begin{equation}\label{eq:Ready_For_Bound_Of_Osc_Int}
\big|  e^{it\Delta}f_{R_j}(x) \big| \lesssim   \frac{R_j^{n - m/2 - 1}}{D_1^{n-m-1}\, D_2^m} \, \big|  e^{it\Delta}g_{R_j}(x_1) \big|.
\end{equation}
Now, from \eqref{eq:Evolution_Of_1_Pre}, write
\begin{equation}\label{eq:Oscillatory_Integral}
\big|  e^{it\Delta}g_{R_j}(x_1) \big| = R_j^{1/2} \, \left|  \int_{\mathbb R}\widehat\phi(\eta) \, e^{2\pi i \,  \lambda_j \theta_j(\eta)}\, d\eta \right|
\end{equation}
where
\begin{equation}
\lambda_j = R_j^{1/2} \, | x_1 + 2tR_j |, \qquad \theta_j(\eta) = \eta + \frac{t}{\lambda_j}\, R_j\, \eta^2, \qquad \theta_j'(\eta) = 1 + 2\frac{tR_j}{\lambda_j}\, \eta.
\end{equation}
Now we exploit the decay of this oscillatory integral. 
Observe that 
\begin{equation}
\begin{split}
\lambda_j & = R_j^{1/2 }\, \Big( 2 t |R_j - R_k| +  \mathcal O \left(  |x_1 + 2tR_k| \right)\,   \Big) \\
& \simeq R_j^{1/2 }\, \Big( 2 \frac{ |R_j - R_k|}{R_k} +  \mathcal O \left( R_k^{-1/2} \right)\,   \Big).
\end{split}
\end{equation}
We separate in two cases:
\begin{itemize}
	\item If $j < k$, then $R_j / R_k \leq 1/2$ and
	\begin{equation}
		\lambda_j \simeq R_j^{1/2 }\, \Big( 1 -   \frac{ R_j}{R_k} +  \mathcal O \left( R_k^{-1/2} \right)\,   \Big)
		\simeq R_j^{1/2}.
	\end{equation}
	In this case, 
	\begin{equation}
	\Big| \frac{t R_j}{\lambda_j} \eta \Big| \leq \frac{ R_j}{R_k \, R_j^{1/2}} \leq \frac{1}{R_j^{1/2}} < \frac14
	\qquad  \Longrightarrow \qquad
	\left| \theta'_j(\eta) \right| > 1/2 > 0.
	\end{equation}

	\item If $j > k$, then $R_j / R_k \geq 2$ and 
	\begin{equation}
	 \lambda_j \simeq R_j^{1/2 }\, \Big(   \frac{ R_j}{R_k} - 1 +  \mathcal O \left( R_k^{-1/2} \right)\,   \Big)
	 \simeq R_j^{1/2 }\, \frac{ R_j}{R_k}.
	\end{equation}
	In particular $\lambda_j > R_j^{1/2}$, so in this case
	\begin{equation}
	\Big| \frac{t R_j}{\lambda_j} \eta \Big| \leq \frac{ R_j}{R_k \,\lambda_j } \simeq \frac{1}{R_j^{1/2}} < \frac14
	\qquad \Longrightarrow \qquad
	\left| \theta'_j(\eta) \right| > 1/2 > 0.
	\end{equation}
\end{itemize}
Thus, in both cases we can integrate by parts in \eqref{eq:Oscillatory_Integral} to obtain  
\begin{equation}
\big|  e^{it\Delta}g_{R_j}(x_1) \big| \lesssim \frac{  R_j^{1/2} }{ \lambda_j^N } \lesssim \frac{1}{R_j^{(N-1)/2}},  \qquad \forall N \in \mathbb N.
\end{equation}
Coming back to \eqref{eq:Ready_For_Bound_Of_Osc_Int}, 
using \eqref{eq:Norm_Of_Initial_Datum} and  recalling that 
$D_1$ and $D_2$ will be powers of $R$, 
we get 
\begin{equation}
 \frac{  \big| e^{it(x)\Delta} f_{R_j} (x) \big|  }{ R_j^{s_m}\, \lVert f_{R_j} \rVert_2  }
 \lesssim \frac{1}{R_j^N}, \qquad \forall N \in \mathbb  N.
\end{equation}
In particular, we get \eqref{eq:Small_Pieces} and the proof is complete.
\end{proof}


\section{Dimension of the set of divergence}\label{sec:Dimension}

In this section we compute the Hausdorff dimension of the divergence set $F$ 
defined in \eqref{eq:F_R} and \eqref{eq:Divergence_Set}.
Recall that the slabs in \eqref{eq:Slabs} are
\begin{equation}\label{eq:slabs}
E_R(p,q) = 
	B^1\bigg( \frac{R}{D_{1}^2} \, \frac{p_1}{q}, \frac{1}{R^{1/2}} \bigg)\times
	B^{n - m -1}\bigg(\frac{1}{D_{1}}\frac{p'}{q}, \frac{1}{R}\bigg)\times
	B^m\bigg(  \frac{p''}{D_{2}}, \frac{1}{R^{1/2}}\bigg),
\end{equation}
and that we build the divergence set with $E_k(p,q) = E_{R_k}(p,q)$. 
Rather than with the parameters $D_1, D_2$ and $Q$, 
we find it more convenient to work with $(u_1,u_2,u_3)$ defined by
\begin{equation}\label{eq:Geometric_Parameters}
R^{u_1} = \frac{QD_1^2}{R}, \qquad 
R^{u_2} = QD_1, \qquad\mbox{and}\qquad
R^{u_3} = D_2,
\end{equation}
or equivalently, 
\begin{equation}
Q = R^{2u_2-u_1-1}, \qquad 
D_1 = R^{1+u_1-u_2}, \qquad\mbox{and}\qquad
D_2 = R^{u_3}.
\end{equation}
In view of \eqref{eq:slabs}, 
$(u_1, u_2, u_3)$ determine the separation of successive slabs 
for each fixed $q$
in the coordinates $x_1, x'$ and $x''$ respectively.

We have a few preliminary restrictions for the parameters.
For each fixed $q$, we want that successive slabs do not intersect with each others.
For that, for instance in $x_1$, we need 
\begin{equation}
\frac{1}{R^{1/2}} < \frac{R}{D_1 q} \simeq \frac{R}{D_1 Q} = \frac{1}{R^{u_1}} \quad \Longrightarrow \quad u_1 \leq 1/2. 
\end{equation}
Also, we require that we have more than a single slab in each of the directions, 
so we require $R^{-u_1} = R/(D_1Q) \ll 1$, which implies $u_1 > 0$. 
Similar reasons suggest that we require
\begin{equation}\label{eq:Restrictions_Basic}
0 < u_1 \le 1/2, \qquad
0 < u_2\le 1 \quad \text{ and } \quad
0 < u_3 \le 1/2.
\end{equation}
Since $Q$ is the size of the denominators $q \in \mathbb N$, we always have $Q \geq 1$, which implies
\begin{equation}\label{eq:Restrictions_Q}
2u_2 - u_1 \ge 1.
\end{equation}

\subsection{Upper bound}
With these restrictions, we can compute an upper bound for $\dim F$.
\begin{prop} \label{thm:upper_bound}
Let $F\subset \R^n$ be the divergence set defined 
in \eqref{eq:F_R} and \eqref{eq:Divergence_Set},
with parameters $(u_1, u_2, u_3)$ as in \eqref{eq:Geometric_Parameters}, 
subject to the restrictions \eqref{eq:Restrictions_Basic} and \eqref{eq:Restrictions_Q}.
Then,
\begin{equation}
\dim F \le \min\{\alpha_1, \alpha_2\},
\end{equation}
where
\begin{equation}
\alpha_1 = \frac{m-1}{2} + (n-m+1)u_2+mu_3
\end{equation}
and
\begin{equation}
\alpha_2 = \begin{cases}
n-m-3 + 4u_2 + 2mu_3, 
	& \text{for} \quad u_2 \le 3/4  \\
n-m+2mu_3, 
	& \text{for} \quad u_2 \ge 3/4 
\end{cases}
\end{equation}
\end{prop}
\begin{proof}
Since $F = \limsup_{k \to \infty} F_k \subset \bigcup_{k\ge N} F_k$ 
for all $N>0$,
it suffices to cover $F_k$ for every $k \in \mathbb N$. 
From the definition in \eqref{eq:F_R}, 
$F_k$ is formed by 
\begin{equation}
Q\cdot R^{u_1}\cdot R^{(n-m-1)u_2}\cdot R^{mu_3} = 
	R^{(n-m+1)u_2 + mu_3 - 1}
\end{equation} 
slabs $E_k(p,q)$.
Each of those slabs is covered by $R_k^{1/2} \, \big(R_k^{1/2}\big)^m$ balls 
of radius $R_k^{-1}$. 
In all, each $F_k$ is covered by 
\begin{equation}
R^{(n-m+1)u_2 + mu_3 - 1} \, R^{\frac{m+1}{2}} = R^{(n-m+1)u_2 + mu_3 + \frac{m-1}{2}},
\end{equation}
so taking $\delta = R_N^{-1}$, we get
\begin{equation}
\mathcal H^\alpha_{R_N^{-1}}(F) \leq \sum_{k=N}^\infty R^{-\alpha} \, R^{(n-m+1)u_2 + mu_3 + \frac{m-1}{2}}.
\end{equation}
Thus, if $\alpha \geq  (n-m+1)u_2 + mu_3 + \frac{m-1}{2} = \alpha_1$,
we get $\mathcal H^\alpha(F) = \lim_{N \to \infty} H^\alpha_{R_N^{-1}} (F) = 0$, 
so $\dim F \leq \alpha_1$. 


To prove $\dim F\le \alpha_2$, we need to arrange 
the slabs of $F$ differently
(see Figure~\ref{fig:Arranging_Slabs} for visual support). 
\begin{figure}
\centering
\includegraphics[width=0.7\textwidth]{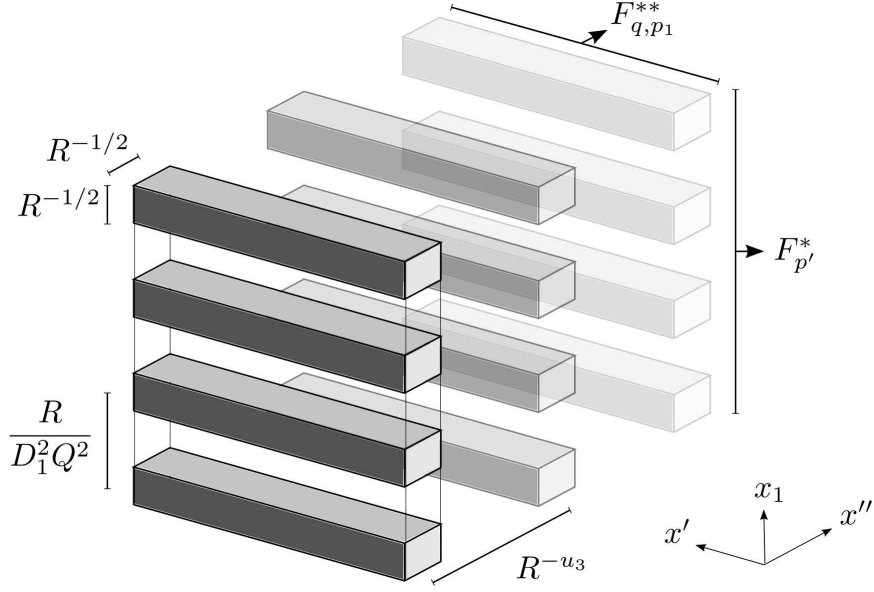}
\caption{Arrangement of the slabs of $F_k$.}
\label{fig:Arranging_Slabs}
\end{figure}
First, observe that in the direction $x''$ 
the slabs are disjoint. 
Thus, it is useful to arrange $F_k$ as  
\begin{equation}\label{eq:Arrangement1}
F_k = \bigcup_{p''\in \Z^m}
	\Bigg[\bigcup_{q\simeq Q}
	\bigcup_{(p_1,p')\in G(q)} E_k(p,q)\Bigg]
  = \bigcup_{p''\in \Z^m} F^*_{k, p''}.
\end{equation}

Let us look at the separation between two slabs 
$E_k(p,q)$ and $E_k(\tilde{p},\tilde{q})$ 
in the direction $x_1$, which is
\begin{equation}
\frac{R}{D_1^2} \frac{| p_1 \tilde{q} - \tilde{p}_1 q |}{q \tilde{q}} \gtrsim \frac{R}{D_1^2 Q^2} = R^{1 - 2u_2}.
\end{equation}
Thus, if we ask $R^{1 - 2u_2} > R^{-1/2}$, which amounts to $u_2 \leq 3/4$, 
the slabs in direction $x_1$ are disjoint.
Consequently, we can further arrange 
\begin{equation}
F^*_{k,p''} = 	\bigcup_{q\simeq Q} \bigcup_{p_1}
	\bigg[\bigcup_{p' : p \in G(q)} 
		E(p,q) \bigg]
	  = 	\bigcup_{q\simeq Q} \bigcup_{p_1} F^{**}_{k,q,p_1, p''},
\end{equation}
and the number of sets $ F^{**}_{k,q,p_1, p''}$ in $F_k$ is at most $R^{2u_2 - 1}\, R^{m\, u_3}$.
Since each set $F^{**}_{k,q,p_1, p''}$ can be covered 
by an 
$R^{-1/2}_k$ neighborhood of a $(n-m-1)$-plane,
in particular we can cover it by  $R^{(n-m-1)/2}_k$ balls of radius 
 $R^{-1/2}_k$.
 Thus, 
 \begin{equation}
 \mathcal H_{R_N^{-1/2}}^\alpha (F) 
 \leq \sum_{k = N}^\infty R_k^{-\alpha / 2} \, R_k^{2u_2 - 1 + mu_3 + (n-m-1)/2},
\end{equation}   
so $ \mathcal H^\alpha(F) = \lim_{N\to \infty}    \mathcal H_{R_N^{-1/2}}^\alpha (F)  = 0$
if $\alpha \geq n - m - 3 + 4u_2 + 2mu_3 = \alpha_2$.
Thus, $\dim F \leq \alpha_2$. 

When $u_2 > 3/4$, the slabs in direction $x_1$ need not be disjoint anymore.
Still, from the arrangement \eqref{eq:Arrangement1},
every $F_{k,p''}^*$ can be covered by a $R_k^{-1/2}$ neighborhood 
of a $(n-m)$-plane, which in turn is covered by $R_k^{(n-m)/2}$ balls
of radius $R_k^{-1/2}$. 
Since there are $R_k^{mu_3}$ different $F_{k,p''}^*$ in $F_k$,
\begin{equation}
 \mathcal H_{R_N^{-1/2}}^\alpha (F) 
 \leq \sum_{k=N}^\infty R_k^{-\alpha/2} \, R_k^{ m u_3 + (n-m)/2 }.
\end{equation}
 Thus, if $\alpha > n-m + 2mu_3 = \alpha_2$, we get 
 $\mathcal H^\alpha (F) = \lim_{N \to \infty}  \mathcal H_{R_N^{-1/2}}^\alpha (F) =0$,
 which implies 
 $\dim F \leq \alpha_2$. 
\end{proof}

\subsection{Lower bound}

As we announced in the introduction, 
to prove the lower bound for $\dim F$ 
we use the Mass Transference Principle from rectangles to rectangles
proved by Wang and Wu \cite{WangWu2021}.
For that, in the following lines we identify our setting with the notation 
and definitions introduced in 
 \cite[Section 3.1]{WangWu2021}.
 
 Let us index each slab $E_k(p,q)$ with $\alpha = (k, p , q)$ and 
gather the indices in 
\begin{gather}\label{eq:Indices_MTP}
J = \bigcup_{k \gg 1} J_k, \\
J_k = \{(k, p, q) \mid Q_k/2 \le q\textrm{ odd} \le Q_k \textrm{ and } (p_1, p', p'') \in G(q) \times \Z^m \}.
\end{gather}
The resonant set $\{\mathcal{R}_\alpha \mid \alpha \in J\}$ 
from \cite[Definition~3.1]{WangWu2021}
corresponds to the set of centers of the slabs, 
so we work with $\kappa = 0$.
Define the function $\beta : J \to \R_+$ by $\beta((k,p,q)) = R_k$, 
and we set $u_k = l_k = R_k$ so that 
$J_k = \{\alpha \mid l_k \le \beta(\alpha) \le u_k \} = \{\alpha \mid  \beta(\alpha) = R_k \} $.
Also, we set $\rho(u) = u^{-1}$, 
 so our slabs can be rewritten as
\begin{equation}\label{eq:Slabs_For_MTP}
E_k(p, q) = B(\mathcal{R}_\alpha, \rho(R_k))^{\vc{b}} = 
	\prod_{i = 1}^n B\big(\mathcal{R}_{\alpha, i}, \rho(R_k)^{b_i}\big), 
\end{equation}
where the exponent $\boldsymbol{b} = (b_1, \ldots, b_n)$ is  
\begin{equation}
\vc{b} = (1/2,\, 
	\underbrace{1,\ldots,1}_{n-m-1},\,
	\underbrace{1/2,\ldots, 1/2}_m). 
\end{equation}
Let us also define the dilation exponent
\begin{equation}
\vc{a} = (a_1,\, 
	\underbrace{a_2,\ldots,a_2}_{n-m-1},\,
	\underbrace{a_3,\ldots, a_3}_m), \qquad \text{ such that } \qquad  a_i \le b_i, \quad \forall i = 1, \ldots, n. 
\end{equation} 
For brevity, most of the time we will just write $\vc{b} = (b_1,b_2,b_3)$
and  $\vc{a} = (a_1,a_2,a_3)$.

We can now adapt the Mass Transference Principle
from rectangles to rectangles in \cite[Theorem~3.1]{WangWu2021} 
to our setting. 
\begin{thm}[Mass Transference Principle from rectangles to rectangles - Theorem 3.1 of \cite{WangWu2021}]
\label{thm:MTP}
Let $\{\mathcal{R}_\alpha \mid \alpha \in J\} \subset \R^n$ be
a set of points. 
Assume that for $(\rho, \vc{a})$ 
there exists $c > 0$ such that for any ball $B$,
\begin{equation} \label{eq:Uniform_local_U}
\mathcal{H}^n \bigg(B \cap \bigcup_{\alpha \in J_k}B \left(\mathcal{R}_\alpha, \rho(R_k) \right)^{\boldsymbol{a}}  \bigg) 
\ge 
	c \, \mathcal{H}^n(B),\qquad
\textrm{for all } k \ge k_0(B),
\end{equation}
where $k_0(B)$ is some constant that depends on the ball $B$. 
Then, for the set
\begin{equation}
W(\vc{b}) = 
	\Big\{x \in \R^n \mid x \in B(\mathcal{R}_\alpha, \rho(R_k))^{\vc{b}} 
	\textrm{ for infinitely many } \alpha \in J \Big\}
\end{equation}
with exponent $\vc{b} = (b_1, \ldots, b_n)$  such that
with $a_i \le b_i$ for all $i= 1, \ldots, n$
we get
\begin{equation}
\dim W(\vc{b}) \ge 
	\min_{B \in \mathcal{B}}\bigg\{
	\sum_{j \in K_1(B)} 1 + 
	\sum_{j \in K_2(B)} \Big(1 - \frac{b_j - a_j}{B}\Big) + 
	\sum_{j \in K_3(B)} \frac{a_j}{B} \bigg\}.
\end{equation}
Here, $\mathcal{B} = \{b_1, \ldots, b_n\}$, and 
for every $B \in \mathcal{B}$ we have 
the
partition of $\{1, \ldots, n\}$ given by
\begin{equation}
\begin{array}{c}
K_1(B) = \{j \mid a_j \ge B\}, \qquad \qquad 
K_2(B) = \{j \mid b_j \le B\}\setminus K_1(B), \\
\\
K_3(B) = \{1, \ldots, n\}\setminus (K_1(B)\cap K_2(B)).
\end{array}
\end{equation}
\end{thm}

\begin{rmk}
As proposed in \cite[Definition 3.3]{WangWu2021},
a system $\{\mathcal{R}_\alpha \mid \alpha \in J\}$ 
that satisfies \eqref{eq:Uniform_local_U} is called 
uniformly locally ubiquitous with respect to $(\rho, \vc{a})$. 
As observed in \cite[Remark 3.2]{WangWu2021}, 
uniform local ubiquity implies that
the $\limsup$ of the dilated slabs has actually full measure.
\end{rmk}

According to \eqref{eq:Indices_MTP} and \eqref{eq:Slabs_For_MTP}, 
 we have
\begin{equation}
\bigcup_{\alpha \in J_k} B \left(\mathcal{R}_\alpha, \rho(R_k) \right)^{\boldsymbol{b}} 
=
\bigcup_{(k,p,q) \in J_k} E_k(p,q)
=
 F_k
\end{equation}
and $W(\vc{b}) = \limsup_{k \to \infty} F_k = F$. 
Thus, to apply Theorem~\ref{thm:MTP} and obtain
a lower bound for $\dim F$, 
we need to find a dilation exponent $\vc{a}$ 
such that the dilated sets
\begin{equation}
F_k^{\vc{a}} = \bigcup_{(k,p,q) \in J_k} E_k^{\vc{a}}(p,q)
\end{equation}
satisfy the uniform local ubiquity condition \eqref{eq:Uniform_local_U}
for every $k \gg 1$. 
To simplify notation, 
we check this for $F_R$ with general $R$ instead of $R_k$. 

First, write $F_R^{\vc{a}}$ as a product 
$F_R^{\vc{a}} = X_R^{a_1, a_2} \times Y_R^{a_3}$ with
\begin{align}
X_R^{a_1,a_2} &= 
	\bigcup_{\substack{Q/2 \le q \le Q \\ q \textrm{ odd} }}\,
	\bigcup_{(p_1,p')\in G(q)} B^1\bigg(  \frac{R}{D_1^2} \, \frac{p_1}{q}, \frac{1}{R^{a_1}}\bigg)\times
	B^{n - m -1}\bigg( \frac{1}{D_1} \,  \frac{p'}{q}, \frac{1}{R^{a_2}}\bigg), \\
Y_R^{a_3} &= 
	\bigcup_{p'' \in \Z^m} B^m\bigg(\frac{p''}{D_2}, \frac{1}{R^{a_3}}\bigg).
\end{align}
Let $B \subset \R^n$ be a ball. Since we always can find a cube inside $B$
with a comparable measure, 
we may assume that $B = B^{n-m} \times B^m$, 
where $B^{n-m}$ and $B^m$
are balls in $\R^{n-m}$ 
and in $\R^m$, respectively.
Then, 
\begin{equation} \label{eq:BXYtoBX}
\begin{split}
\mathcal{H}^n(B\cap ( X_R^{a_1,a_2} \times Y_R^{a_3}) )
& = \mathcal{H}^n( (B^{n-m}\cap  X_R^{a_1,a_2} ) \times (B^m \cap Y_R^{a_3}) ) \\
& \simeq 
	\mathcal{H}^{n-m}(B^{n-m} \cap X_R^{a_1,a_2}) \, \mathcal{H}^m(B^m \cap Y_R^{a_3}).
\end{split}
\end{equation}
Let us first estimate $\mathcal{H}^m(B^m \cap Y_R^{a_3})$. 
Since $D_2 = R^{u_3} \to \infty$ when $R \to \infty$, 
for large enough $R$
there are approximately $D_2^m\, \mathcal H^m(B^m)$ slabs of $Y_R^{a_3}$ 
in the ball $B^m$. 
Thus, 
\begin{equation}
\mathcal{H}^m (B^m \cap Y_R^{a_3}) \simeq D_2^m \, \mathcal H^m(B^m) \, R^{-ma_3} = R^{m(u_3 - a_3 )}.
\end{equation}
Thus, 
\begin{equation}\label{eq:Dilation_3}
a_3 = u_3 \qquad \Longrightarrow \qquad \mathcal{H}^m (B^m \cap Y_R^{a_3}) 
\simeq \mathcal{H}^m (B^m ).
\end{equation}

\begin{figure}[t]
\centering
\includegraphics[width=0.9\textwidth]{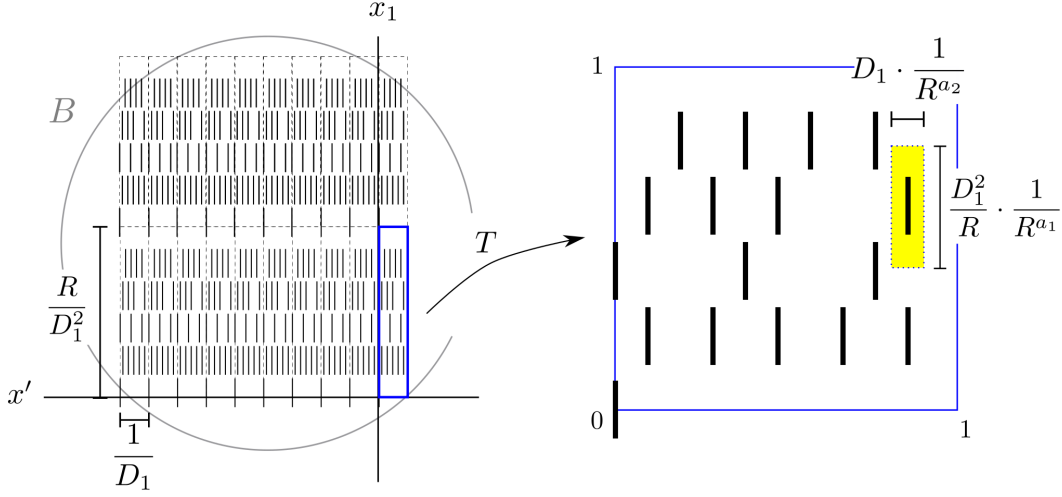}
\caption{In blue, the unit cell $\widetilde{X}_R^{a_1,a_2}$.
On the right, $\Omega_R^{a_1,a_2} = T(\widetilde{X}_R^{a_1,a_2})$.
In black on the right, the image by $T$ of the original slabs, 
and in yellow, 
the image of the slabs dilated by $a_1,a_2$.
To apply the Mass Transference Principle,
we must prove that $\Omega_k^{\vc{a}}$ covers a positive
portion of the unit cell.}
\label{fig:Zoom_Set}
\end{figure}

Regarding $\mathcal{H}^n(B\cap X_R^{a_1,a_2})$,
the set $X_k^{a_1,a_2}$ has periodic a structure, 
as shown in Figure~\ref{fig:Zoom_Set}.
Indeed, under the shrinking condition
\begin{equation}\label{eq:Restriction_Shrinking}
\frac{R}{D_1^2} \ll 1 \quad\iff\quad
	 u_2 - u_1 < \frac{1}{2},
\end{equation}
the set $X_k^{a_1,a_2}$ is a union of copies of 
the unit cell
\begin{equation}
\widetilde{X}_R^{a_1,a_2} =  
	\bigcup_{ \substack{ Q/2 \le q \le Q \\ q \textrm{ odd} } }\,
	\bigcup_{(p_1,p')\in G(q)\cap [0, q)^{n - m}} 
	B^1\bigg(  \frac{R}{D_1^2} \, \frac{p_1}{q}, \frac{1}{R^{a_1}}\bigg)\times
	B^{n - m -1}\bigg( \frac{1}{D_1} \,  \frac{p'}{q}, \frac{1}{R^{a_2}}\bigg),
\end{equation}
which we mark in blue in Figure~\ref{fig:Zoom_Set}. 
In this situation, 
the number of unit cells in a ball $B^{n-m}$ 
is approximately $\mathcal{H}^n(B)D_1^{n - m - 1}D_1^2/R$, 
so
\begin{equation} \label{eq:preliminar_BX}
\mathcal{H}^{n-m}(B^{n-m}\cap X_R^{a_1,a_2}) \simeq 
	\mathcal{H}^{n-m}(B^{n-m}) \, \frac{D_1^{n - m + 1}}{ R } \, \mathcal{H}^{n-m}(\widetilde{X}_R^{a_1,a_2})
\end{equation}
To compute $\mathcal{H}^{n-m}(\widetilde{X}_R^{a_1,a_2})$,
we use the  transformation $T : (x_1,x') \mapsto (D_1^2x_1/R, D_1x')$, 
which sends $\widetilde{X}_R^{a_1,a_2}$ to the set
\begin{equation} \label{eq:Dilation_UC}
\Omega_R^{a_1,a_2} = 
	T\widetilde{X}_R^{a_1,a_2} = 
	\bigcup_{ \substack{ Q/2 \le q  \le Q \\ q \textrm{ odd} }}\,
	\bigcup_{(p_1,p')\in G(q)\cap [0, q)^{n - m}} 
	B^1\bigg(\frac{p_1}{q}, \frac{D_1^2}{R^{1+a_1}}\bigg)\times
	B^{n - m -1}\bigg(\frac{p'}{q}, \frac{D_1}{R^{a_2}}\bigg).
\end{equation}
Since $\mathcal{H}^{n-m}(\widetilde{X}_R^{a_1,a_2}) 
= \mathcal{H}^{n-m}(\Omega_R^{a_1,a_2}) \, R/D_1^{n - m + 1}$, then
from \eqref{eq:BXYtoBX}, \eqref{eq:Dilation_3} and \eqref{eq:preliminar_BX}
we see that 
\begin{equation}
\begin{split}
\mathcal{H}^n(B\cap F_R^{\vc{a}} )
& \simeq 
\mathcal{H}^m (B^m) \, \mathcal{H}^{n-m} (B^{n-m}) \, \mathcal{H}^{n-m}(\Omega_R^{a_1,a_2}) \\
& \simeq \mathcal H^n (B) \, \mathcal{H}^{n-m}(\Omega_R^{a_1,a_2}).
\end{split}
\end{equation}
Thus, having chosen $a_3 = u_3$, 
to verify \eqref{eq:Uniform_local_U}
it suffices to find $a_1,a_2$ such that 
\begin{equation}\label{eq:Objective_Omega_Measure}
\mathcal{H}^{n-m}(\Omega_R^{a_1,a_2}) \ge c > 0, \qquad
\textrm{for } R \gg 1.
\end{equation}
To do so, we use a lemma from \cite{Pierce2021}.
\begin{lem}[Lemma~4.1 of \cite{Pierce2021}] 
\label{thm:almost_disjoint_union}
Let $J$ be a finite set of indices and
$\{I_j\}_{j \in J}$ be a collection of measurable sets in $\R^n$.
Suppose that these sets have comparable size,
that is, $B_0 \le \abs{I_j} \le B_1$ for all $j \in J$, and
that they are regularly distributed in the sense that
\begin{equation} \label{eq:thm:almost_disjoint}
\abs{\{(j,j')\in J \times J \mid 
	I_j\cap I_{j'} \neq \emptyset\}} \le C\abs{J}.
\end{equation} 
Then,
\begin{equation}
\Big|\bigcup_{j \in J} I_j \Big| \ge 
	\frac{B_0}{B_1C}\sum_{j \in J} \abs{I_j}.
\end{equation}
\end{lem}

With the aid of Lemma~\ref{thm:almost_disjoint_union},
we adapt  \cite[Lemma~4.2]{Pierce2021} to estimate 
the measure of $\Omega_R^{a_1,a_2}$.
\begin{lem}
\label{thm:Minkowski}
Let $Q \gg 1$, $t_1, t_2 \ge 1$  and $\Omega \subset \R^N$ defined as
\begin{equation}
\Omega = 
	\bigcup_{\substack{ Q/2 \le q \le Q \\ q\textrm{ odd}}}\,
	\bigcup_{\substack{ (p_1, p') \in [0,q)^N \\ \operatorname{gcd}(p_1,q)=1 }}
	B^1\bigg(\frac{p_1}{q}, \frac{1}{Q^{t_1}}\bigg)\times
	B^{N -1}\bigg(\frac{p'}{q}, \frac{1}{Q^{t_2}}\bigg).
\end{equation}  
If
\begin{equation} \label{eq:thm:Minkowski}
t_1 + (N - 1)t_2 = N + 1, 
\end{equation}
there exists $c > 0$ such that
$
\mathcal{H}^N(\Omega) \ge c > 0.
$
\end{lem}

\begin{proof}
We apply Lemma \ref{thm:almost_disjoint_union} with
\begin{equation}
J = \{(p_1, p', q) \mid Q/2 \le q \le Q, \quad q\textrm{ odd}, \quad p \in [0,q)^N \, \,  
	\textrm{ and }  \, \, \operatorname{gcd}(p_1,q) = 1\, \}
\end{equation}
and 
\begin{equation}
I_{p_1, p', q} = B^1\bigg(\frac{p_1}{q}, \frac{1}{Q^{t_1}}\bigg)\times
	B^{N -1}\bigg(\frac{p'}{q}, \frac{1}{Q^{t_2}} \bigg).
\end{equation}
By the hypothesis \eqref{eq:thm:Minkowski}, 
$\abs{I_{p_1, p', q}} = Q^{-t_1 - (N-1)t_2} = Q^{- (N + 1)}$. 
Thus, the first hypothesis of Lemma~\ref{thm:almost_disjoint_union} 
is satisfied with $B_0 = B_1$.
On the other hand, 
the size of the index set is
\begin{equation}
\abs{J} = 
	\sum_{\substack{ Q/2 \le q \le Q \\ q\textrm{ odd} }} \varphi(q) \, q^{N - 1}.
\end{equation}
We use the formula $\varphi(q) = q\sum_{d \mid q}\mu(d)/d$, where
$\mu$ is the Möbius function \cite[Sec.~16.3]{hardyWright2008}, to write
\begin{align}\label{eq:Using_Mobius}
\abs{J} &= 
\sum_{\substack{ Q/2 \le q \le Q \\ q\textrm{ odd} }} q^N \, 
	\sum_{d \,  \mid \,  q} \frac{\mu(d)}{d} 
\quad 
\simeq
\quad 
 Q^N \,  \sum_{d \in \mathbb N}\frac{\mu(d)}{d}
	\sum_{\substack{ Q/2 \le q \le Q \\ q\textrm{ odd} }}
		\ind_{\{d \, \mid \,  q\}} \\
&= Q^N  \, \sum_{d \in \mathbb N, \, d \textrm{ odd}}\frac{\mu(d)}{d}
	\sum_{ \substack{ Q/(2d) \le k \le Q/d \\  k \textrm{ odd}}} 1 \\
\end{align}
If $Q$ is large enough, then 
\begin{equation}
|J|   \simeq Q^N \sum_{d \in \mathbb N, \, d \textrm{ odd}} \frac{\mu(d)}{d} \, 
	\frac{Q}{d}  
	\quad  = \quad  Q^{N + 1}\sum_{d \in \mathbb N, \, d \textrm{ odd}} \frac{\mu(d)}{d^2} 
	 \quad  \simeq \quad Q^{N + 1},
\end{equation}
where the last sum is finite because $\mu(d) \in \{-1, 0, 1\}$ for all $d \in \N$. 
Hence, to apply Lemma~\ref{thm:almost_disjoint_union}
we have to prove 
\begin{equation}
\abs{\{(j, j')\in J \times J \,  \mid \, I_j\cap I_{j'} \neq \emptyset\}} \lesssim Q^{N + 1}.
\end{equation}

First, the diagonal contribution of equal indices $j = j'$ is $\abs{J}$,
so it is enough to prove
\begin{equation} \label{eq:almost_disjoint_pq}
\abs{\{(j, j')\in J \times J \mid j \neq j' \textrm{ and }
	I_j\cap I_{j'} \neq \emptyset\}} \lesssim Q^{N + 1}.
\end{equation}
To prove \eqref{eq:almost_disjoint_pq},
let us first fix $q$ and $\widetilde{q}$ and
count all $j = (p_1, p', q)$ and $j'= (\widetilde{p}_1, \tilde{p}', \widetilde{q})$ 
such that $I_j\cap I_{j'} \neq \emptyset$.
In this case, 
\begin{equation} \label{eq:intersection_condition}
\Big|\frac{p_1}{q} - \frac{\widetilde{p}_1}{\widetilde{q}}\Big| < \frac{2}{Q^{t_1}} \qquad \mbox{and} \qquad
\Big|\frac{p_l}{q} - \frac{\widetilde{p}_l}{\widetilde{q}}\Big| < \frac{2}{Q^{t_2}}, \quad
	l = 2, \ldots, N
\end{equation} 
There are two cases:
\begin{itemize}
	\item \textbf{Case $\boldsymbol{q = \widetilde{q}}$.} 
From \eqref{eq:intersection_condition} and $t_1,t_2 \ge 1$ we have that
$0 < \abs{p_l - \widetilde{p}_l} < 2$ for all $l$. 
Thus, for each $q$, 
we can pick $\lesssim Q^N$ pairs $(j, j')$.
Summing over all odd $q$, 
the total contribution is of the order of $Q^{N + 1}$.

	\item \textbf{Case $\boldsymbol{q \neq \widetilde{q}}$.}
	From \eqref{eq:intersection_condition} we have that
\begin{equation} \label{eq:intersection_q_diff}
\abs{\widetilde{q}p_1 - q\widetilde{p}_1} \le 2Q^{2 - t_1} \qquad \mbox{and} \qquad
\abs{\widetilde{q}p_l - q\widetilde{p}_l} \le 2Q^{2 - t_2}, \quad
	l = 2, \ldots, N.
\end{equation}
Let us fix $l = 1,\ldots, N$
and count the number of 
$0\le p_l < q$ and $0 \le \widetilde{p}_l < \widetilde{q}$
that satisfy \eqref{eq:intersection_q_diff}.
Let  $d = \operatorname{gcd}(q, \widetilde{q})$ 
and write $q = sd$, $\widetilde{q} = \widetilde{s}d$ 
such that  $\operatorname{gcd}(s, \widetilde{s}) = 1$.

Call $m = \widetilde{q}p_l - q\widetilde{p}_l$.
We want to count the number of ways we can write $m$ like that,
that is, how many  $0 \leq r_l < q$ and $0 \leq \widetilde{r}_l < \widetilde{q}$
satisfy $\widetilde{q}p_l - q\widetilde{p}_l$?
We would have 
$\widetilde{s}p_l - s\widetilde{p}_l = \widetilde{s}r_l - s\widetilde{r}_l$, 
which implies $s \mid p_l - r_l$,
or equivalently $0 \le r_l = p_l + ks < q$ for some $k \in \N$.
The last inequality can hold at most for $d$ different values of $k$.
Thus, we can write $m$ in at most $d$ different ways.

On the other hand, necessarily $d \mid m$. 
Since $\abs{m} \le 2Q^{2 - t_i}$,
we can work with at most $2Q^{2-t_i}/d$ values of $m$. 
Since each of them can be written in $d$ different ways, 
we conclude that the number of pairs 
$p_l$ and $\widetilde{p}_l$ satisfying \eqref{eq:intersection_q_diff} 
is at most $ Q^{2 - t_i}$.
 
Since $l = 1$ goes with $t_1$ and $l = 2, \ldots, N$ go with $t_2$, 
for each fixed $q$ and $\widetilde{q}$ the number of $p$ and $\widetilde{p}$ 
is at most $Q^{2 - t_1}\, Q^{(2-t_2)(N-1)} = Q^{N-1}$.
Finally, summing over all different $q$ and $\widetilde{q}$ gives 
a total contribution of the order of $Q^{N+1}$. 

%
%

\end{itemize}
The two cases together prove \eqref{eq:almost_disjoint_pq}. 
Thus, we can use Lemma~\ref{thm:almost_disjoint_union} and write
\begin{align}
\mathcal{H}^n(\Omega) &\gtrsim 
	\sum_{\substack{ Q/2 \le q \le Q  \\ q\textrm{ odd}  } }\,
	\sum_{\substack{ (p_1, p') \in [0,q)^N \\ \operatorname{gcd}(p_1,q)=1 }} \abs{I_{p_1, p', q}} 
\quad \simeq \quad  \sum_{\substack{ Q/2 \le q \le Q  \\ q\textrm{ odd}  } }
	\frac{ \varphi(q)\, q^{N - 1} }{ Q^{t_1 + (N - 1)t_2}}  \\
	& 
	\simeq \frac{1}{Q^2} \, \sum_{\substack{ Q/2 \le q \le Q  \\ q\textrm{ odd}  } } \varphi(q) 
	\quad \simeq \quad  1,
\end{align}
where the last equality follows proceeding like in \eqref{eq:Using_Mobius}.
\end{proof}

With Lemma~\ref{thm:Minkowski}
we get the conditions that we need for $a_1$ and $a_2$ 
in order to have \eqref{eq:Objective_Omega_Measure}.
\begin{lem}\label{thm:Lemma_Restriction_a}
For the parameters $u_1, u_2$ satisfying the restrictions
\eqref{eq:Restrictions_Basic}, 
\eqref{eq:Restrictions_Q} and \eqref{eq:Restriction_Shrinking}, 
let $a_1$ and $a_2$ be such that 
\begin{equation}\label{eq:thm:a_restriction_basic}
u_1 \le a_1 \qquad \mbox{and} \qquad u_2 \le a_2, 
\end{equation}
and 
\begin{equation}\label{eq:thm:a_restriction}
a_1 + (n - m - 1)a_2 = (n - m + 1)u_2 - 1 
\end{equation}
Then, there exists $c >0$ such that
\begin{equation} \label{eq:thm:dilation_UC_Large}
\mathcal{H}^{n - m}(\Omega_R^{a_1,a_2}) \ge c > 0,  \qquad
\forall R \gg 1.
\end{equation}
\end{lem}

\begin{proof}
For $\Omega_R^{a_1,a_2}$, 
which we defined in \eqref{eq:Dilation_UC}, 
we want to apply Lemma~\ref{thm:Minkowski} 
with $N = n - m$ and 
\begin{equation}
\frac{1}{Q^{t_1}} = \frac {D_1^2}{ R^{1 + a_1 }} =  \frac{1}{R^{a_1 - 1 - 2(u_1 - u_2)}} \qquad \mbox{and} \qquad
\frac{1}{Q^{t_2}} = \frac{D_1}{ R^{a_2}} =  \frac{1}{R^{a_2 - 1 - (u_1 - u_2)}}.
\end{equation}
For that, we need $Q  = R^{2u_2 - u_1 - 1} \gg 1$, 
which means $2u_2 - u_1 - 1 > 0$. 
In that case, we get 
\begin{equation}\label{eq:t1t2_a1a2}
 t_1 = \frac{a_1 - 1 - 2(u_1 - u_2)}{ 2u_2 - u_1 - 1 }, 
 \qquad  t_2 = \frac{ a_2 - 1 - (u_1 - u_2)}{ 2u_2 - u_1 - 1 }.
\end{equation} 
The condition $t_1 \geq 1$ implies $a_1 \geq u_1$, 
while $t_2 \geq 1$ implies $a_2 \geq u_2$.
On the other hand, 
replacing \eqref{eq:t1t2_a1a2} in \eqref{eq:thm:Minkowski} we get
the condition
\begin{equation}
a_1 + (n - m - 1) a_2 = (n-m+1) u_2 - 1,
\end{equation}
under which there exists $c>0$ such that 
$\mathcal H^{n-m} (\Omega_R^{a_1,a_2 )} \geq c$. 

According to restriction \eqref{eq:Restrictions_Q}, we are only left with the case 
$2u_2 - u_1 - 1 = 0$, which corresponds to $Q = 1$.
In this case, the set turns into 
\begin{equation}\label{eq:Dilation_UC_2}
\Omega_R^{a_1,a_2} = 
	B^1\bigg(0, \frac{D_1^2}{R^{1+a_1}}\bigg)\times
	B^{n - m -1}\bigg(0, \frac{D_1}{R^{a_2}}\bigg),
\end{equation}
so we get 
$\mathcal H^{n-m} (\Omega_R^{a_1,a_2 )} \geq c > 0$
if we ask $D_1^2 = R^{1+a_1}$ and $D_1 = R^{a_2}$. 
This amounts to $a_1 = u_1$ and $a_2 = u_2$. 
Observe that \eqref{eq:thm:a_restriction} is also satisfied in this case. 
\end{proof}

\begin{rmk}\label{rmk:Restrictions}
The restrictions we found for  the parameters $(u_1, u_2, u_3)$
 and the dilation exponents $(a_1, a_2, a_3)$
are the following:

For the parameters, from \eqref{eq:Restrictions_Basic}, 
\eqref{eq:Restrictions_Q} and \eqref{eq:Restriction_Shrinking} we have
\begin{equation}\label{eq:Restrictions_Parameters}
0 < u_1, u_3 \leq 1/2, \qquad 0 < u_2 \leq 1, \qquad 2u_2 - u_1 \geq 1, \qquad u_2 - u_1 < 1/2.
\end{equation}
In particular, $u_2 > 1/2$. Regarding $(a_1, a_2, a_3)$, we got 
\begin{equation}\label{eq:Restrictions_Dilation}
u_1 \leq a_1 \leq 1/2, \qquad u_2 \leq a_2 \leq 1, \qquad a_3 = u_3, \qquad a_1 + (n-m-1)a_2 = (n-m+1)u_2 - 1.
\end{equation}
From the last restriction in \eqref{eq:Restrictions_Dilation}
together with $a_1 \leq 1/2$ and $a_2 \leq 1$ we get the additional restriction
\begin{equation}\label{eq:Restrictions_Parameters_2}
(n-m+1) u_2 \leq   n - m + 1/2.
\end{equation}
Thus, for $(u_1, u_2, u_3)$ that satisfy \eqref{eq:Restrictions_Parameters} 
and \eqref{eq:Restrictions_Parameters_2},
we can always find $(a_1, a_2, a_3)$ that satisfy \eqref{eq:Restrictions_Dilation}, 
so \eqref{eq:Objective_Omega_Measure} holds
and we can use the Mass Transference Principle. 

The restrictions for $(u_1, u_2, u_3)$ are shown in Figure~\ref{fig:parameter_u}.
\end{rmk}

\begin{figure}
\centering
\includegraphics[width=0.9\textwidth]{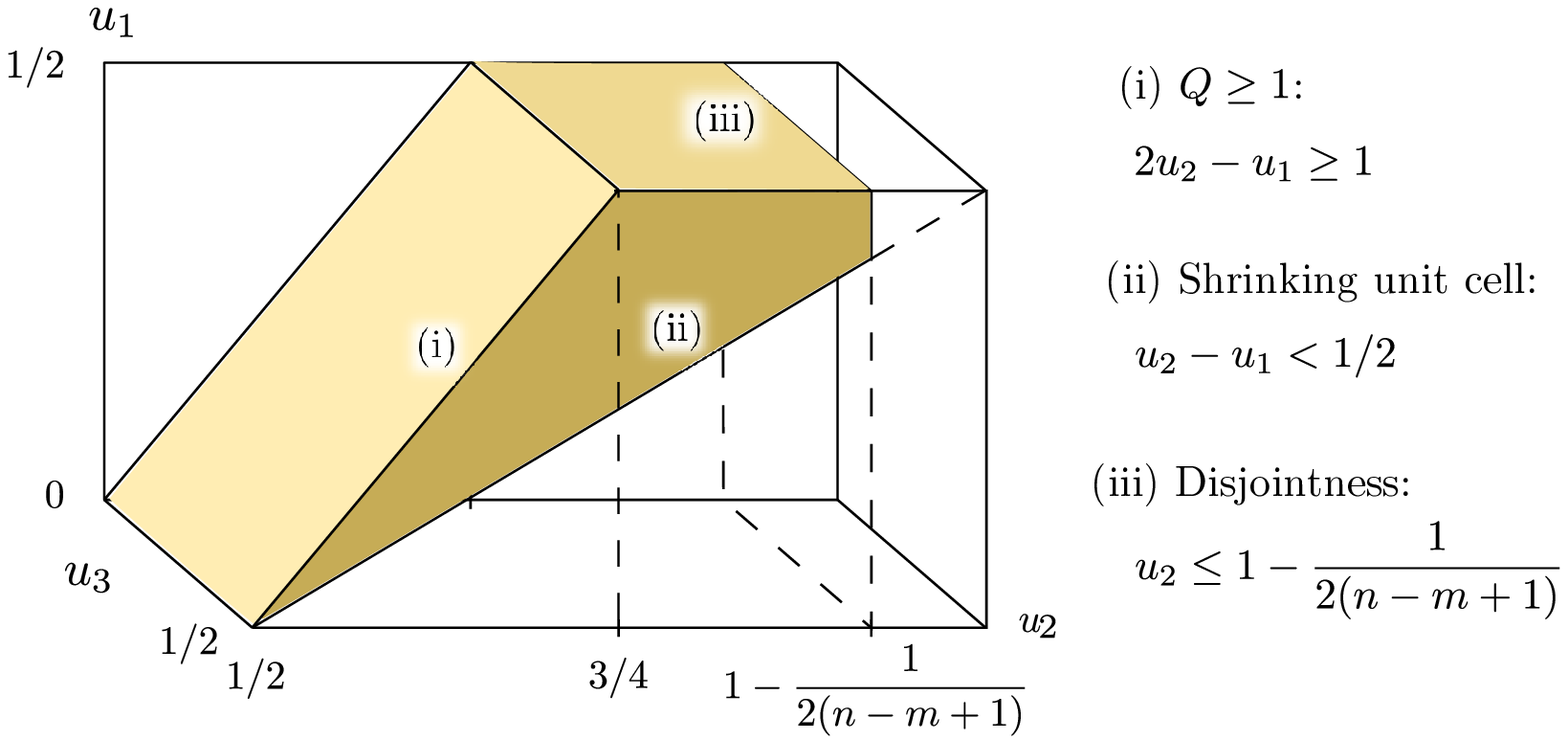}
\caption{Restrictions on $\vc{u} = (u_1, u_2, u_3)$.}
\label{fig:parameter_u}
\end{figure}

According to Remark~\ref{rmk:Restrictions}
we can apply the Mass Transference Principle in Theorem~\ref{thm:MTP}. 
With it, we show that the upper bound given in Proposition~\ref{thm:upper_bound} is sharp. 
\begin{prop}\label{thm:Dimension}
Let $F\subset \R^n$ be the divergence set defined 
in \eqref{eq:F_R} and \eqref{eq:Divergence_Set},
with parameters $(u_1, u_2, u_3)$ as in Remark~\ref{rmk:Restrictions}. 
Then, 
\begin{equation}
\dim F = \min\{\alpha_1, \alpha_2\},
\end{equation}
where
\begin{equation}
\alpha_1 = \alpha_1(u_2, u_3) =  \frac{m-1}{2} + (n-m+1)u_2+mu_3
\end{equation}
and 
\begin{equation}
\alpha_2 = \alpha_2 (u_2, u_3) = \begin{cases}
n-m-3 + 4u_2 + 2mu_3, 
	& \mbox{for} \quad u_2 \le 3/4  \\
n-m+2mu_3, 
	& \mbox{for} \quad u_2 \ge 3/4.
\end{cases}
\end{equation}
\end{prop}

\begin{proof}
By Lemma~\ref{thm:upper_bound}, we only have to prove the lower bound.

For any $(u_1, u_2, u_3)$ as in Remark~\ref{rmk:Restrictions}, 
we can find $(a_1, a_2, a_3)$ that satisfy \eqref{eq:Restrictions_Dilation}.
Then,  Lemma~\ref{thm:Lemma_Restriction_a} proves \eqref{eq:Objective_Omega_Measure}, 
which allows us to use the Mass Transference Principle in Theorem~\ref{thm:MTP}.   
Since $\boldsymbol{b} = (1/2, 1, 1/2)$, then 
\begin{equation} \label{eq:MTP_Applied}
\dim F \ge \min_{B \in\{1,1/2\}}
	\bigg\{\sum_{j\in K_1(B)} 1 + 
	\sum_{j\in K_2(B)} \Big(1 - \frac{b_j - a_j}{B}\Big) +
	\sum_{j\in K_3(B)} \frac{a_j}{B} \bigg\},
\end{equation}
where
\begin{equation}
\begin{array}{c}
K_1(B) = \{j \mid a_j \ge B\}, \qquad \qquad 
K_2(B) = \{j \mid b_j \le B\}\setminus K_1(B), \\
\\
K_3(B) = \{1, \ldots, n\}\setminus (K_1(B)\cap K_2(B)).
\end{array}
\end{equation}
We compute each term in the minimum \eqref{eq:MTP_Applied} separately:
\begin{itemize}
	\item $B = 1$: If $a_2 < 1$, we get
$K_1(1) = \emptyset$, $K_2(1) = \{1,\ldots,n\}$ and $K_3(1) = \emptyset$,
so the term in braces is
\begin{equation} \label{eq:dim_1}
	n - \sum_{j = 1}^n b_j + \sum_{j = 1}^n a_j \\
	= \frac{m + 1}{2} + 
		a_1 + (n - m - 1)a_2 + ma_3.
\end{equation}
Replacing  \eqref{eq:Restrictions_Dilation} above, we get
\begin{equation}\label{eq:dim_11}
\alpha_1 = 
\alpha_1(u_2, u_3) = 
	\frac{m - 1}{2} + (n - m + 1)u_2 + mu_3.
\end{equation}
If $a_2 = 1$, then $K_1(1) = \{ 2, \ldots, n-m  \}$, $K_2(1) = \{1, n - m + 1, \ldots, n\}$
and $K_3(1) = \emptyset$. Thus, we get
\begin{equation}
(n - m - 1) + (m+1) - (1/2 - a_1) - m(1/2 - a_3) = n - \frac{m+1}{2} + a_1 + ma_3,
\end{equation}
This is equal to \eqref{eq:dim_1}, so we get the same $\alpha_1$.

	\item 
$B = 1/2$: Let us first assume that $a_1, a_3 < 1/2$ so that
$K_1(1/2) = \{2,\ldots, n-m\}$, $K_2(1/2) = \{1, n -m + 1,\ldots,n\}$ and $K_3(1/2) = \emptyset$.
The term in braces is thus
\begin{align}
\alpha_2( u_3, a_1) &= 
	(n - m - 1) + (m + 1) - 
	(1 - 2a_1) - m(1 - 2a_3)\\
	&= n - m - 1 + 2a_1 + 2mu_3. \label{eq:dim_2}
\end{align}
In the case that $a_1 < 1/2$ and that $a_3 = 1/2$, we have
$K_1(1/2) = \{2,\ldots, n\}$, $K_2(1/2) = \{1\}$ and $K_3(1/2) = \emptyset$,
and we get
\begin{align}
\alpha_2(a_1) &= 
(n-1) + 1 - (1 - 2a_1) = n - 1 + 2a_1,
\end{align}
which is the same as \eqref{eq:dim_2} because $u_3 = a_3 = 1/2$. 
Similarly, the cases $a_1 = 1/2$, $a_3 < 1/2$
and $a_1 =  a_3 = 1/2$ yield the same result.
\end{itemize}
Joining the two expressions for the minimum, we get
\begin{equation} \label{eq:dim_Lower_B_raw}
\dim F \ge 
	\min\{\alpha_1(u_2, u_3), \alpha_2(u_3, a_1)\}, \qquad \forall a_1 \text{ like in } \eqref{eq:Restrictions_Dilation}.
\end{equation}
Thus, we want to choose the value of $a_1$ that gives the largest $\alpha_2(u_3, a_1)$. 

According to \eqref{eq:dim_2}, we need to take the largest possible $a_1$.
Since $a_1 \leq 1/2$, in principle we may take $a_1 = 1/2$. 
In view of \eqref{eq:Restrictions_Dilation}, that implies 
\begin{equation}\label{eq:Dilation_Case_1}
(n-m+1) u_2 = 3/2 + (n-m-1) a_2.
\end{equation}
However, we need $a_2 \geq u_2$, which under the 
restriction \eqref{eq:Dilation_Case_1} is equivalent to 
$u_2 \geq 3/4$. 
Thus, we separate two cases:
\begin{itemize}
	\item  If $u_2 \le 3/4$,  then $a_1=1/2$ is admissible, so 
	the maximum for $\alpha(u_3, a_1)$ is
	\begin{equation}\label{eq:dim_2_High}
	\alpha_2(u_3) = n - m + 2mu_3, \qquad \text{ if } u_2 \geq 3/4.
	\end{equation}
	
	\item If $u_2 < 3/4$, then $a_1 = 1/2$ is not admissible, 
	because $a_2 < u_2$. Then, the largest admissible value for $a_1$
	corresponds to $a_2 = u_2$, which in view of \eqref{eq:Restrictions_Dilation} gives
	$a_1 = 2u_2 - 1$. Thus, the maximum $\alpha_2$ is 
	\begin{equation}\label{eq:dim_2_Low}
	\alpha_2(u_2, u_3) = n - m - 3 + 4u_2 + 2m u_3, \qquad \text{ if } u_2 < 3/4. 
	\end{equation}		
\end{itemize}
Consequently, from \eqref{eq:dim_Lower_B_raw}, we obtain
\begin{equation}
\dim F \ge 
	\min\{\alpha_1(u_2, u_3), \alpha_2(u_2, u_3)\},
\end{equation}
where $\alpha_1$ is defined in \eqref{eq:dim_11}
and $\alpha_2$ is defined in \eqref{eq:dim_2_High} and \eqref{eq:dim_2_Low}.
The proof is complete. 
%
\end{proof}

\subsection{The case $\boldsymbol{m=n-1}$}
\label{sec:Dim_Degenerate}

The counterexample is not as interesting in this case because 
the Talbot effect is absent.
We discuss it briefly. 
The set of divergence is actually much simpler, given by
\begin{gather}
F = \limsup_{k \to \infty} F_k
= \limsup_{k \to \infty}	
	\bigcup_{p\in \Z^m} E_k(p) 
		\label{eq:nT_divergence_set} \\
E_k(p) = 
	[-1,0]\times
	B^{n-1}\bigg(\frac{p}{D_{2,k}}, R_k^{-1/2}\bigg),
	\label{eq:noT_slabs}
\end{gather}
so only the parameter $D_2 = R^{u_3}$ survives. 
We use the Mass Transference Principle Theorem~\ref{thm:MTP}
in $\mathbb R^{n-1}$ with $\boldsymbol{a} = (a, \ldots, a)$,
 which corresponds to the original version in 
\cite{BeresnevichVelani2006}. 
The dilation $a$ needed for the local ubiquity condition \eqref{eq:Uniform_local_U}
must satisfy $D_2 = R^{\frac{a}{2(n-1)}}$, that is, $a = 2(n-1)u_3$, 
which implies $\dim F = 1 + 2(n-1)u_3$. 
The dimension can also be computed 
 using the methods in
Section~8.2 of \cite{falconer_book03}.


\section{Sobolev Regularity} \label{sec:Sobolev_Regularity}

We begin by recalling that the Sobolev regularity $s_m = s_m(Q, D_1, D_2)$ of the counterexample 
was given in \eqref{eq:Heuristic_Exponent} by
\begin{equation}
R^{s_m} =  R^{1/4} \, \left( \frac{R}{D_1 Q} \right)^{(n-m-1)/2} \, \left( \frac{R^{1/2}}{D_2} \right)^{m/2 }.
\end{equation}
Using \eqref{eq:Geometric_Parameters}, 
we rewrite it in terms of the geometric parameters $(u_1, u_2, u_3)$ as
\begin{equation}\label{eq:Sobolev_Regularity_2}
s_m(u_2, u_3) = 
	\frac{2n-m-1}{4} - \frac{n-m-1}{2}u_2 - \frac{mu_3}{2}.
\end{equation}

Given a fixed dimension $\alpha$, 
we want to maximize $s_m$.
As we showed in Proposition~\ref{thm:Dimension}, 
the dimension of the divergence set is a function $\alpha(u_2, u_3)$,
so we are imposing the restriction $\alpha = \alpha(u_2, u_3)$. 
This still leaves one degree of freedom $v$ in $s_m(\alpha, v)$,
which we might set either as $u_2$ or as $u_3$.
Let us denote the maximum regularity by $s_m(\alpha) = \max_v s_m(\alpha, v)$.

The case $m=0$ corresponds 
to the counterexample studied 
in \cite{LucaPonceVanegas2021},
which gives the regularity
\begin{equation}
s_m(\alpha) = \frac{n}{2(n+1)} + \frac{n-1}{2(n+1)}(n-\alpha),
\end{equation}
so we focus on $m \geq 1$. 
Fix $\dim F = \alpha$. 
By Proposition~\ref{thm:Dimension}, 
\begin{equation}\label{eq:Alpha}
\alpha = \min \{ \alpha_1(u_2, u_3), \alpha_2(u_2, u_3) \}, 
\end{equation}
where 
\begin{gather}\label{eq:Alphas}
\alpha_1 = \frac{m-1}{2} + (n-m+1)u_2+mu_3
\quad\mbox{and}\quad
\alpha_2 = \begin{cases}
n-m-3 + 4u_2 + 2mu_3, 
	& \mbox{for } u_2 \le 3/4  \\
n-m+2mu_3, 
	& \mbox{for } u_2 \ge 3/4.
\end{cases}
\end{gather}
This is a restriction on $(u_2, u_3)$, which takes the form of
a broken line in the $(u_2, u_3)$ plane.
We want to pick a point $(u_2, u_3)$ that gives the maximum  $s_m(u_2, u_3)$. 
In the arguments that follow,
we suggest the reader to use
Figures~\ref{fig:regularity_max__m_small}, \ref{fig:regularity_max_m_medium} 
and \ref{fig:regularity_max_m_large} as visual support.

According to the restrictions in Remark~\ref{rmk:Restrictions}, 
we have 
\begin{equation}\label{eq:Domain_For_u2u3}
(u_2, u_3) \in \mathcal D = \left[ \, \frac12, \,  1  - \frac{1}{2(n-m+1)} \right] \times \left[ 0, 1/2 \right].
\end{equation}
Let us first determine  in $\mathcal D$
the boundary between the two lines in \eqref{eq:Alpha}.
If $u_2 \geq 3/4$, 
	\begin{equation}
	\alpha_1 \leq \alpha_2 \quad \Longleftrightarrow \quad (n-m+1) u_2 - m u_3  \leq n - \frac{3m}{2} + \frac12,
	\end{equation}
so the boundary is
	\begin{equation}\label{eq:Boundary_Line_1}
	(n-m+1) u_2 - m u_3  = n - \frac{3m}{2} + \frac12, \qquad \text{ when } u_2 \geq 3/4. 
	\end{equation} 
This line crosses the points 
	\begin{equation}
	(u_2,u_3) = \left( 1 - \frac{1}{2(n-m+1)}, \, \frac12  \right) \quad \text{ and } \quad (u_2,u_3) = \left( 1 - \frac{m+1}{2(n-m+1)}, \, 0 \right), 
	\end{equation}		
	so it is completely in $\mathcal D \cap \{ u_2 \geq 3/4 \}$ if 
	\begin{equation}\label{eq:Line_Is_On_The_Right}
	1 - \frac{1}{2(n-m+1)} \geq \frac34  \quad \Longleftrightarrow \quad m \leq \frac{n-1}{3}.
	\end{equation}
This shows that we need to separate cases for $m$,
and it will become evident that we also need to study
the cases $n-3 \leq m \leq n-1$ separately.

\subsection{When $\boldsymbol{m < n-3}$ and $\boldsymbol{m \leq (n-1)/3}$}
\label{sec:SubsectionFirst}

This case is displayed in Figure~\ref{fig:regularity_max__m_small}.
According to \eqref{eq:Line_Is_On_The_Right}, the boundary line \eqref{eq:Boundary_Line_1} 
is completely included in $u_2 \geq 3/4$. 
This suggests that in $u_2 \le 3/4$ we always have $\alpha_1 \leq \alpha_2$. 
Indeed, 
\begin{equation}
\begin{split}
\alpha_1 \leq \alpha_2 \, & \Longleftrightarrow \, \frac{m-1}{2} + (n-m+1)\, u_2 + mu_3 \leq n-m-3 + 4u_2 + 2mu_3 \\
& \Longleftrightarrow (n-m-3)\, u_2 - mu_3 \leq n - \frac{3m}{2} - \frac52,
\end{split}
\end{equation}
and together with $u_2 \leq 3/4$ and $u_3 \geq 0$, 
the condition $m \leq (n-1)/3$ allows us to write
\begin{equation}
(n-m-3)\, u_2 - mu_3 \leq \frac{3}{4}\, (n-m-3) \leq n - \frac{3m}{2} - \frac52.
\end{equation}
Thus, when $u_2 \le 3/4$ we have $\min\{\alpha_1, \alpha_2\} = \alpha_1$.

Let us compute the Sobolev regularity:
\begin{itemize}
	\item  In the region where $\min\{ \alpha_1, \alpha_2\} = \alpha_2$, since $u_2 \geq 3/4$, 
we may write
\begin{equation}\label{eq:alpha2_m_small}
\alpha = n - m + 2mu_3  \quad \Longrightarrow \quad mu_3 = \frac{m - (n-\alpha)}{2}.
\end{equation}
Consequently, $u_3$ is fixed. Replacing in \eqref{eq:Sobolev_Regularity_2}, we get
\begin{equation}\label{eq:s_m_small_alpha2}
s_m(\alpha, u_2) =  \frac{n-\alpha + 1}{4} + \frac{n-m-1}{2}(1 - u_2).
\end{equation}
Thus, to maximize $s_m(\alpha, u_2)$ we need to minimize $u_2$. 
This is attained on the boundary \eqref{eq:Boundary_Line_1}.

	\item In the region where $\min\{ \alpha_1, \alpha_2\} = \alpha_1$ we have
\begin{equation}\label{eq:alpha1_m_small}
\alpha = \frac{m-1}{2} + (n-m+1)u_2 + mu_3, 
\end{equation}
so replacing in \eqref{eq:Sobolev_Regularity_2} we get
\begin{equation}\label{eq:s_m_small}
s_m(\alpha, u_2) = \frac{n - 1 - \alpha}{2} + u_2.
\end{equation}
In this case, to maximize $s_m(\alpha, u_2)$ we need to maximize $u_2$. 
The maximum $u_2$ in this region may be either
on the boundary \eqref{eq:Boundary_Line_1} 
or in $u_3=0$.
\end{itemize}

\begin{figure}[h]
\centering
\includegraphics[width=0.75\linewidth]{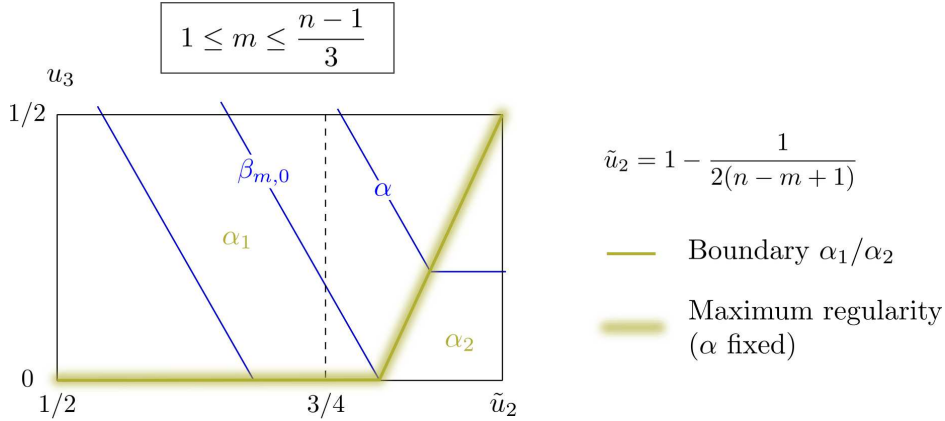}
\caption{For fixed $\alpha$, the maximum regularity is attained on the blurred, green line.} 
\label{fig:regularity_max__m_small}
\end{figure}

Thus, in all cases, given a dimension $\alpha$, 
the maximum $s_m(\alpha) = \max_{u_2} s_m(\alpha, u_2)$ is attained either 
on the boundary \eqref{eq:Boundary_Line_1} 
or on $u_3 =0$. 
Let $\beta_{m,0}$ be
the dimension where this transition happens, that is,
the value $\beta_{m,0}$ such that the line
$\beta_{m,0} = \min\{\alpha_1, \alpha_2\}$ crosses
the intersection of the boundary \eqref{eq:Boundary_Line_1} and $u_3=0$. 
When $\beta_{m,0} = \alpha_2$, we are always in $u_2 \geq 3/4$, 
so we may write $\beta_{m,0} = \alpha_2 = n-m+2mu_3$. 
Since the point of intersection has $u_3 = 0$, we deduce that
\begin{equation}
	\beta_{m,0} = n-m.
\end{equation}
This generates two different cases for $\alpha$:

\begin{itemize}
	\item If $\alpha \le \beta_{m,0}$, then $\alpha = \min\{\alpha_1, \alpha_2\} = \alpha_1(u_2, u_3)$ 
	and the maximum $s_m(\alpha)$ is attained at $u_3=0$. 
	Thus, from \eqref{eq:alpha1_m_small} and \eqref{eq:s_m_small},
	\begin{equation}\label{eq:s_m_small_alpha_small}
	s_m(\alpha)  = \frac{n}{2(n-m+1)} + \frac{n-m-1}{2(n-m+1)} (n-\alpha), \qquad \text{ if } \quad n/2 \leq \alpha \leq n-m. 
	\end{equation}
	Observe that the smallest possible $\alpha$ corresponds to 
	$\alpha = \alpha_1(u_2, u_3)$ crossing the point $(u_2,u_3) = (1/2,0)$, which gives $\alpha_{\text{min}} = n/2$. 

	\item If $\beta_{m,0} \leq \alpha \leq n$, the maximum of $s_m(\alpha,u_2)$
	is attained on the boundary \eqref{eq:Boundary_Line_1}. 
	The intersection between the broken line $\alpha  = \min \{\alpha_1, \alpha_2\}$
	 and the boundary \eqref{eq:Boundary_Line_1} is determined by
	  \begin{equation}
	  (n-m+1)u_2 = n - m - \frac{n-\alpha - 1}{2}
	  \qquad \text{and} \qquad
	  2mu_3 = m - (n-\alpha).
	  \end{equation}
	  Replacing this point either in \eqref{eq:s_m_small_alpha2} 
	  or in \eqref{eq:s_m_small}, we get
	  \begin{equation}\label{eq:s_m_small_alpha_large}
	  s_m(\alpha) = \frac{n-m}{2(n-m+1)} + \frac{n-m}{2(n-m + 1)}(n-\alpha), \qquad \text{ if } \quad n-m \leq \alpha \leq n.
	  \end{equation}
	\end{itemize}

\subsection{When $\boldsymbol{m < n-3}$ and  $\boldsymbol{(n-1)/3 < m \leq n/2 - 1}$}

We display this case at the left of Figure~\ref{fig:regularity_max_m_medium}. 
Now the boundary \eqref{eq:Boundary_Line_1} crosses $u_2 \geq 3/4$
while in $\mathcal D$, and the crossing point is
\begin{equation}\label{eq:Point_Of_Boundary_Line_1}
(u_2, u_3) = \Big(\frac{3}{4}, \, \frac12 - \frac{n-m-1}{4m}\Big).
\end{equation} 
In this case, $\min\{ \alpha_1,\alpha_2 \}$ also changes in $u_2 \le 3/4$
and the boundary is given by 
\begin{equation}\label{eq:Boundary_line_2}
\alpha_1 = \alpha_2 \quad \Longleftrightarrow \quad (n-m-3)u_2 - mu_3 = n - \frac{3m}{2} - \frac52.
\end{equation}
This line has positive slope as long as $m < n-3$, and
it passes through the points \eqref{eq:Point_Of_Boundary_Line_1} and 
\begin{equation}\label{eq:Point_Of_Boundary_Line_2}
(u_2, u_3) = \Big(1 - \frac{m-1}{ 2\, (n-m-3) }, \, 0\Big).
\end{equation}
It makes a difference whether the point \eqref{eq:Point_Of_Boundary_Line_2}
is in $\mathcal D$ or not. 
One immediately sees that 
%
\begin{equation}\label{eq:Point_2_In_D}
\eqref{eq:Point_Of_Boundary_Line_2} \in \mathcal D \quad \Longleftrightarrow \quad 1 - \frac{m-1}{ 2\, (n-m-3) } \geq \frac12 \quad \Longleftrightarrow \quad m \leq \frac{n}{2} - 1,
\end{equation}
which is the case we are considering now.

\begin{rmk}\label{rmk:Optimizing}
We saw in Subsection~\ref{sec:SubsectionFirst} that:
\begin{itemize}
	\item When $\min\{\alpha_1,\alpha_2\} = \alpha_1$ 
	we need to maximize $u_2$. 
	\item When $\min\{\alpha_1,\alpha_2\} = \alpha_2$ and $u_2 \geq 3/4$ 
	we need to minimize $u_2$.
\end{itemize}
Now we have an additional case:
\begin{itemize}
	\item When $\min\{\alpha_1,\alpha_2\} = \alpha_2$ and $u_2 \le 3/4$ we have 
	$\alpha = n - m - 3 + 4u_2 + 2mu_3$, so replacing in \eqref{eq:Sobolev_Regularity_2} we get
	\begin{equation}\label{eq:s_m_middle}
	s_m(\alpha, u_2) = \frac{n-m-2}{2} + \frac{n-\alpha}{4} - \frac{n-m-3}{2}\,  u_2.
	\end{equation}
	Since $m < n-3$, to maximize $s_m(\alpha, u_2)$ we need to minimize $u_2$. 
\end{itemize}
\end{rmk}
Consequently, depending on the value of $\alpha$, 
the maximum of $s_m(\alpha, u_2)$ is attained in the boundary
\eqref{eq:Boundary_Line_1}, in the boundary \eqref{eq:Boundary_line_2} or in $u_3=0$.
Let us determine which $\alpha$ corresponds to each case. 
\begin{itemize}
	\item The interval corresponding to the boundary line \eqref{eq:Boundary_Line_1} 
	is $\alpha \in [\beta_{m,2},n]$, where $\beta_{m,2}$ is such that the  broken line
	$\beta_{m,2} = \min\{ \alpha_1, \alpha_2\}$ 
	crosses the point \eqref{eq:Point_Of_Boundary_Line_1}.
	Thus,
	\begin{equation} \label{eq:beta_m2}
	\beta_{m,2} = n - m + 2m \left(   \frac12 - \frac{n-m-1}{4m} \right) =  \frac{n + m + 1}{2}
\end{equation}	 
	The analysis in this case is identical to that in \eqref{eq:s_m_small_alpha_large}, so
	\begin{equation}
	s_m(\alpha) = \frac{n-m}{2(n-m+1)} + \frac{n-m}{2(n-m+1)}(n-\alpha), \qquad \beta_{m,2} \leq \alpha \leq n.
	\end{equation}
	
	\item The interval corresponding to the boundary \eqref{eq:Boundary_line_2},
	which is in $u_2 \le 3/4$, 
	is $\alpha \in [\beta_{m,1}, \beta_{m,2}]$, where 
	the broken line $\beta_{m,1} = \min \{ \alpha_1, \alpha_2 \}$ 
	crosses the point \eqref{eq:Point_Of_Boundary_Line_2}.
	This means that
	\begin{equation} \label{eq:beta_m1}
	\beta_{m,1} = n - m - 3 + 4 \left( 1 - \frac{m-1}{2(n-m-3)} \right) 
	= n - (m-1) \, \frac{n-m-1}{n-m-3}.
	\end{equation}
	For $\alpha \in [\beta_{m,1}, \beta_{m,2}]$, the point $(u_2,u_3)$ of the 
	broken line $\alpha = \min\{ \alpha_1, \alpha_2\}$ that is in
	the boundary line \eqref{eq:Boundary_line_2} has
	 \begin{equation}
	u_2 = \frac{n + \alpha - 2(m+1) }{2(n-m-1)}.
	\end{equation}
	Thus, the Sobolev regularity we get from \eqref{eq:s_m_middle} is
	\begin{equation}\label{eq:s_m_mid_alpha_mid}
	s_m(\alpha) = \frac12 +   \frac{n-m-2}{2(n-m-1)}(n-\alpha), \qquad \beta_{m,1} \leq \alpha \leq \beta_{m,2}.
	\end{equation}

	\item For the last interval $\alpha \in [\alpha_{\text{min}}, \beta_{m,1}]$
	 we have $\alpha = \alpha_1(u_2, u_3)$, 
	 and the maximum $u_2$ is attained at $u_3=0$. 
	 The procedure is the same as in \eqref{eq:s_m_small_alpha_small}, so 
	 we get 
	\begin{equation}
	s_m(\alpha)  = \frac{n}{2(n-m+1)} + \frac{n-m-1}{2(n-m+1)} (n-\alpha), \qquad n/2 \leq \alpha \leq \beta_{m,1}. 
	\end{equation}	 	 
	 
	\end{itemize}

\begin{figure}[h]
\centering
\includegraphics[width=0.9\linewidth]{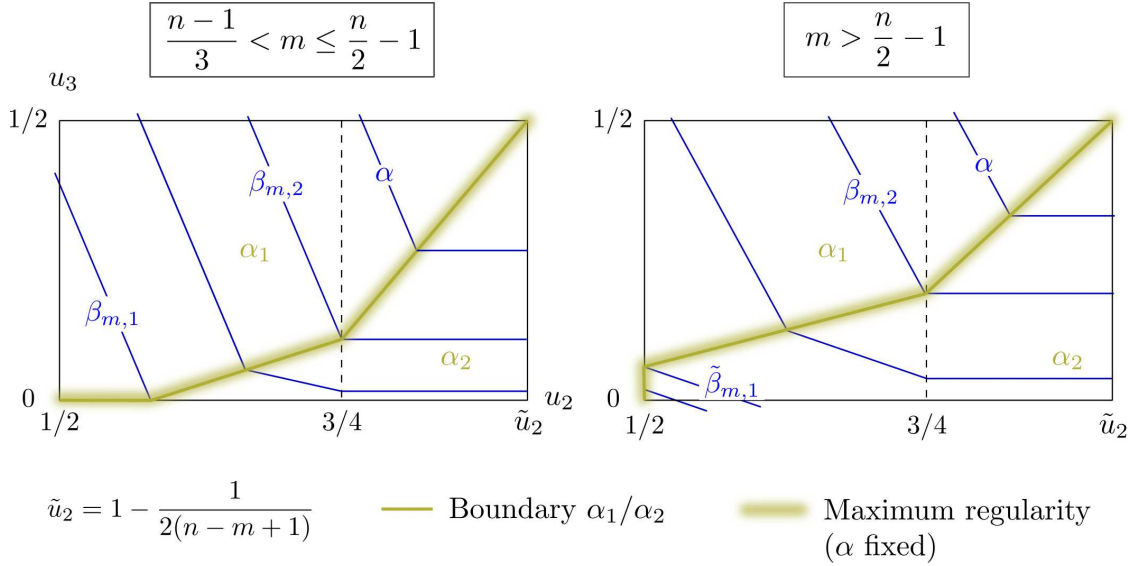}
\caption{For fixed $\alpha$, the maximum regularity is attained on the blurred, green line.} 
\label{fig:regularity_max_m_medium}
\end{figure}

\subsection{When $\boldsymbol{n/2 - 1 < m < n-3}$}

This case is shown at the right of Figure~\ref{fig:regularity_max_m_medium}. 
By \eqref{eq:Point_2_In_D}, we have that 
\eqref{eq:Point_Of_Boundary_Line_2} $\notin \mathcal D$. 
The useful point in this case is the intersection of the boundary \eqref{eq:Boundary_line_2}
with $u_2=1/2$, that is,
\begin{equation}\label{eq:Point_Of_Boundary_Line_2_m_large}
u_2 = \frac12
\qquad\mbox{and}\qquad
u_3 = \frac{m+1-n/2}{m} > 0. 
\end{equation}
In this case, depending on $\alpha$ and again following Remark~\ref{rmk:Optimizing}, 
the maximum $s_m(u_2, u_3)$ is found in the boundary \eqref{eq:Boundary_Line_1}, 
in the boundary \eqref{eq:Boundary_line_2} or on the line $u_2= 1/2$.
Let us determine the ranges for $\alpha$ in each case:
\begin{itemize}
	\item For the interval corresponding to the boundary \eqref{eq:Boundary_Line_1}, 
	the analysis is exactly the same as in the previous case, so we get
	 \begin{equation}
	s_m(\alpha) = \frac{n-m}{2(n-m+1)} + \frac{n-m}{2(n-m+1)}(n-\alpha), \qquad \beta_{m,2} \leq \alpha \leq n.
	\end{equation}
	
	\item The interval corresponding to the boundary \eqref{eq:Boundary_line_2} is 
	$\alpha \in [\widetilde{\beta}_{m,1}, \beta_{m,2}]$,
	 where $\widetilde{\beta}_{m,1} = \min\{\alpha_1, \alpha_2\}$ crosses 
	the point \eqref{eq:Point_Of_Boundary_Line_2_m_large}. 
	Evaluating in $\alpha_1$, we get
	\begin{equation}
	\widetilde{\beta}_{m,1} = \frac{m-1}{2} + \frac{n-m+1}{2} + m+1-\frac{n}{2} = m+1.
	\end{equation}
	For $\alpha \in [\widetilde{\beta}_{m,1}, \beta_{m,2}]$,
	the analysis is the same as in \eqref{eq:s_m_mid_alpha_mid}, so we get
	\begin{equation}
	s_m(\alpha) = \frac12 +   \frac{n-m-2}{2(n-m-1)}(n-\alpha), \qquad \widetilde{\beta}_{m,1} \leq \alpha \leq \beta_{m,2}.
	\end{equation}

	\item The last interval is $\alpha \in [\alpha_{\text{min}}, \widetilde\beta_{m,1}]$, 
	where the maximum is attained in $u_2 = 1/2$. 
	In this case,  $\alpha_{\text{min}}$ is such that 
	$\alpha_{\text{min}} = \min \{ \alpha_1, \alpha_2 \}$ crosses the point $(u_2,u_3) = (1/2, 0)$, 
	that is, 
	\begin{equation}
	\alpha_{\text{min}} = n - m - 1.
	\end{equation}
	Thus, for $\alpha \in [ \alpha_{\text{min}}, \widetilde\beta_{m,1}]$ 
	we have $\alpha = \alpha_2(u_2, u_3)$,
	so replacing  $u_2 = 1/2$ in \eqref{eq:s_m_middle} we get
	\begin{equation}
	s_m(\alpha) = \frac{n-m-1}{4} + \frac{n-\alpha}{4}, \qquad n - m - 1 \leq \alpha \leq \widetilde\beta_{m,1}.
	\end{equation}
	
\end{itemize}

\subsection{When $\boldsymbol{m = n - 3}$}

As shown in Figure~\ref{fig:regularity_max_m_large}, 
the boundary \eqref{eq:Boundary_line_2} 
is now the  horizontal line
\begin{equation}\label{eq:BoundaryHorizontal}
mu_3 = - n + \frac{3m}{2} + \frac52.
\end{equation}
\begin{itemize}
	\item The first interval $\alpha \in [\beta_{m,2}, n]$ does not change with respect to the previous cases:
	\begin{equation}
	s_m(\alpha) = \frac{n-m}{2(n-m+1)} + \frac{n-m}{2(n-m+1)}(n-\alpha), \qquad \beta_{m,2} \leq \alpha \leq n.
	\end{equation}
	
	\item The rest $\alpha < \beta_{m,2}$ are unified in this case. 
	This is because when $u_2 \leq 3/4$
we have $\alpha_2(u_2, u_3) = n - m - 3 + 4u_2 + 2mu_3 = 4u_2 + 2mu_3$.
Thus, for $(u_2, u_3)$ such that 
$\alpha = \alpha_2(u_2, u_3)$, 
the regularity is 
\begin{equation}\label{eq:s_m_n-3_alpha_small}
s_m(u_2, u_3) = \frac{2n - m - 1}{4} - \frac{mu_3}{2} - \frac{n - m - 1 }{2} u_2 
 = \frac{2n-m-1-\alpha}{4}, 
\end{equation}
which is independent of $u_2$ and $u_3$. 
That means that when $\alpha = \alpha_2(u_2, u_3)$ and $u_2 \leq 3/4$, 
all $u_2$ give the same $s_m$. 
In particular, $s_m(u_2, u_3)$ is the same both
in the boundary \eqref{eq:BoundaryHorizontal} and in $u_2 = 1/2$, 
so 
	\begin{equation}
	s_m(\alpha) =  \frac12 + \frac{n-\alpha}{4}, \qquad 2 \leq \alpha \leq \beta_{m-2} = n-1.
	\end{equation}
As in the previous case, $\alpha_{\text{min}} = n-m-1 = 2$. 
\end{itemize}

Observe that this case matches the result of the case $n/2 - 1 \leq m < m-3$.

\begin{figure}[h]
\centering
\includegraphics[width=0.9\linewidth]{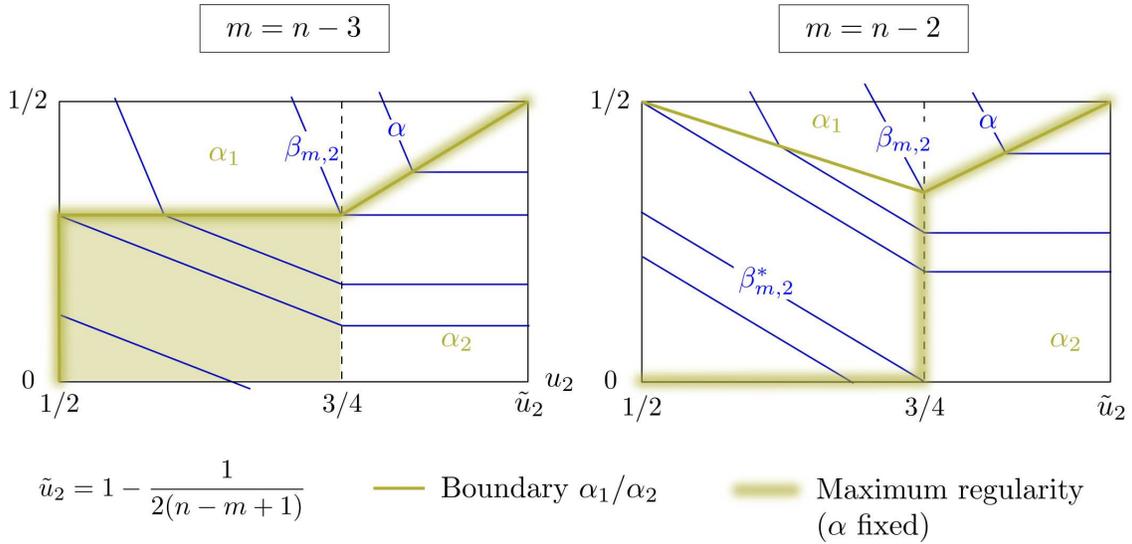}
\caption{For fixed $\alpha$, the maximum regularity is attained on the blurred, green zone.} 
\label{fig:regularity_max_m_large}
\end{figure}

\subsection{When $\boldsymbol{m = n - 2}$}

This case corresponds to the right of Figure~\ref{fig:regularity_max_m_large}. 
Now the boundary \eqref{eq:Boundary_line_2} 
in $u_2 \le 3/4$ takes the form
\begin{equation}
u_2 + (n-2)u_3 = \frac{n-1}{2}.
\end{equation}
It crosses the point $(u_2,u_3) = (1/2,1/2)$, and its slope is $-1/(n-2)$.
Observe that the slope of $\alpha = \alpha_1(u_2,u_3)$ is $-3/(n-2)$, 
while that of $\alpha = \alpha_2(u_2, u_3)$ in $u_2 \le 3/4$ is $-2/(n-2)$. 

When $\alpha = \alpha_1(u_2, u_3)$ we still have \eqref{eq:s_m_small}, 
so we want to maximize $u_2$.
However, when $\alpha = \alpha_2(u_2, u_3)$ and $u_2 \le 3/4$, 
the regularity we computed in \eqref{eq:s_m_middle}
takes the form
\begin{equation}\label{eq:s_m_n-2_alpha_mid}
s_m(\alpha) = \frac{n-\alpha}{4} + \frac{u_2}{2},
\end{equation} 
so we want to maximize $u_2$. 
When $\alpha = \alpha_2(u_2, u_3)$ and $u_2 \geq 3/4$ 
the regularity is \eqref{eq:s_m_small_alpha2},
so we still want to minimize $u_2$. 
Thus, depending on $\alpha$,  the maximum of $s_m(u_2, u_3)$ is attained on 
the boundary \eqref{eq:Boundary_Line_1} in $u_2 \geq 3/4$,
on the line $u_2 = 3/4$ or on the line $u_3=0$. 
We classify $\alpha$ accordingly:

\begin{itemize}
	\item As in all previous cases, in the interval $\alpha \in [\beta_{m,2}, n]$ the result is
	\begin{equation}
	s_{n-2}(\alpha) = \frac{1}{3} + \frac{n-\alpha}{3}, \qquad n-\frac12  \leq \alpha \leq n.
	\end{equation}
	
	\item The second interval is now $\alpha \in [\beta^*_{m,1}, \beta_{m, 2}]$,
	where $\beta^*_{m,1}$ 
	corresponds to the point $(u_2,u_3) = (3/4,0)$, that is, 
	\begin{equation}
	\beta^*_{m,1}  = n - m + 2mu_3 =2.
	\end{equation}
	In this case, the maximum is on $u_2 = 3/4$, so
	from \eqref{eq:s_m_n-2_alpha_mid} we get
	\begin{equation}
	s_{n-2}(\alpha) = \frac38 + \frac{n-\alpha}{4}, \qquad 2 \leq \alpha \leq \beta_{m,2} = n-1/2.
	\end{equation}
	
	\item The last interval is $\alpha \in [\alpha_{\text{min}}, \beta^*_{m,1}]$,
	and as in the previous cases $\alpha_{\text{min}} = n-m-1 = 1$. 
	Now, the maximum is attained at $u_3=0$, 
	and thus, $\alpha = \alpha_2(u_2, u_3) = 4u_2 - 1$. 
	Replacing this in \eqref{eq:s_m_n-2_alpha_mid}
	we get
	\begin{equation}
	s_{n-2}(\alpha) = \frac{n+1}{8} + \frac{n-\alpha}{8}, \qquad 1 \leq \alpha \leq 2.
	\end{equation}
	
\end{itemize}

\subsection{When $\boldsymbol{m = n - 1}$}

From \eqref{eq:Sobolev_Regularity_2}, we have
\begin{equation}
s_{n-1}(u_3) = \frac{n}{4} - \frac{n-1}{2}\, u_3.
\end{equation}
From the dimension in Subsection~\ref{sec:Dim_Degenerate}, 
we have $\alpha = 1 + 2(n - 1)u_3$, 
where we can pick any $0 < u_3 \le 1/2$.
Thus, the regularity is
\begin{equation}
s_{n-1}(\alpha) = \frac{1 + n - \alpha}{4},
	\qquad 1\le \alpha \le n.
\end{equation}

\subsection{Summary of the results of this subsection}
Let us gather the results we got by defining
\begin{gather}
s_{3,m}(\alpha) = \frac{n}{2(n-m+1)} + \frac{n-m-1}{2(n-m+1)}(n-\alpha), \\
s_{4,m}(\alpha) = \frac{n-m}{2(n-m+1)} + \frac{n-m}{2(n-m+1)}(n-\alpha), \\
s_{5,m}(\alpha) = \frac12 + \frac{n-m-2}{2(n-m-1)}(n-\alpha),
\end{gather}
and also, from \eqref{eq:beta_m1} and \eqref{eq:beta_m2}, 
\begin{equation}
\beta_{m, 1} = n - (m-1) \,  \frac{n-m-1}{n-m-3}
\qquad \text{and} \qquad 
\beta_{m, 2} = \frac{n+m+1}{2}.
\end{equation}


\begin{prop} \label{thm:SobolevRegularity}
Let $0 \leq m \leq n-1$ and $s_m(\alpha)$ as below. 
Then, for every $s < s_m(\alpha)$,
there exists $f \in H^s(\R^n)$ 
such that $e^{it\Delta} f$ diverges in
a set of dimension $\alpha$.

The exponent $s_m(\alpha)$ is as follows.
For $0 \leq m \leq n-3$:
\begin{enumerate}[(i)]
	
		\item\label{it:m_small} If $0 \le m \le (n - 1)/3$, 
\begin{equation} \label{eq:thm:m_small}
s_m(\alpha) = \begin{cases}
s_{3, m}(\alpha), & 
	\qquad  n/2 \le \alpha \le n - m, \\
s_{4, m}(\alpha), &
	 \qquad  n - m \le \alpha \le n.
\end{cases}
\end{equation}

		\item\label{it:m_medium} If $(n - 1)/3 < m \le n/2 - 1$, 
		\begin{equation}\label{eq:thm:m_medium}
s_m(\alpha) = \begin{cases}
s_{3, m}(\alpha), & 
	\qquad n/2 \le \alpha \le \beta_{m, 1}, \\
s_{5, m}(\alpha), &
	\qquad \beta_{m, 1} \le \alpha \le \beta_{m, 2}, \\
s_{4, m}(\alpha), &
	\qquad \beta_{m, 2} \le \alpha \le n.
\end{cases}
\end{equation}
		
		\item\label{it:m_large} If $n/2 - 1 < m \le n - 3$, 
		\begin{equation}\label{eq:thm:m_large}
s_m(\alpha) = \begin{cases}
s_{3, m}(\alpha), & 
	\qquad n - m - 1 \le \alpha \le m + 1, \\
s_{5, m}(\alpha), &
	\qquad m + 1 \le \alpha \le \beta_{m, 2}, \\
s_{4, m}(\alpha), &
	\qquad \beta_{m, 2} \le \alpha \le n.
\end{cases}
\end{equation}
\end{enumerate}
On the other hand, if $m = n - 2$, then
\begin{equation}
s_{n-2}(\alpha) = \begin{cases}
\dfrac{n+1}{8} + \dfrac{n-\alpha}{8}, & 
	\qquad 1 \le \alpha \le 2, \\[6pt]
\dfrac38 + \dfrac{n-\alpha}{4}, &
	\qquad 2 \le \alpha \le n - 1/2, \\[6pt]
\dfrac{1}{3} + \dfrac{n-\alpha}{3}, &
	\qquad  n - 1/2 \le \alpha \le n, 
\end{cases}
\end{equation}
and if $m = n - 1$, then
\begin{equation}
s_{n-1}(\alpha) = \frac{1 + n - \alpha}{4},
	\qquad 1\le \alpha \le n.
\end{equation}

\end{prop}

The reader may want to compare the first two cases in Proposition~\ref{thm:SobolevRegularity}
with Lemma~3.2 in \cite{Du2020}; 
in Du's paper replace 
$d$ by $n + 1$,
$m$ by $m + 1$, and
$\kappa_i$ by $s_{i, m}$.


\section{Maximum Regularity}
	\label{sec:Maximum_Regularity}

For each $0 \leq m \leq n-1$, 
Proposition~\ref{thm:SobolevRegularity} gives 
the regularity $s_m(\alpha)$ for the counterexample. 
Thus, we immediately get the following theorem.
\begin{thm}\label{thm:Short}
Let $n/2 \leq \alpha \leq n$. 
For every $0 \leq m \leq n-1$, and for  $s_m(\alpha)$
as in Proposition~\ref{thm:SobolevRegularity}, 
define
\begin{equation} \label{eq:Grand_Maximum}
s(\alpha) = 
	\max_{0\le m \le n - 1} s_m(\alpha).
\end{equation}
Then, for $s < s(\alpha)$
there exists $f \in H^s(\R^n)$ 
such that $e^{it\Delta} f$ diverges in
a set of dimension $\alpha$.
\end{thm}

Our aim in this section is to dissect  this quantity. 
First, 
we show that in the maximum \eqref{eq:Grand_Maximum} it suffices 
to consider small $m$.
\begin{lem}
Let $m_1 = \floor{n/2 - 1}$. Then,
\begin{equation} \label{eq:Maximum_to_m_small}
s(\alpha) = \max_{0\le m \le m_1} s_m(\alpha)
\end{equation}
In particular, 
\begin{equation}
s(\alpha) = s_0(\alpha), \qquad \text{ when } n=2, 3. 
\end{equation}
\end{lem}

\begin{proof}
The objective is to discard the contribution of every $m > m_1$ to the maximum. 
For that, we are going to prove that $s_m(\alpha) \leq s_0(\alpha)$ 
for $n/2 \leq \alpha \leq n$. 

First observe that for $\alpha = n$, 
$s_m(n) \le s_0(n)$ holds for all $m$. 
Thus, we may work with $\alpha < n$. 
We now study each $m $ separately. 

\begin{itemize}
	\item For $m=n-1$, from Proposition~\ref{thm:SobolevRegularity} we have
$s_{n-1}(\alpha) = (1 + n - \alpha)/4$ for every $1 \leq \alpha \leq n$. 
Since $s_{n-1}(n/2) \le s_0(n/2) = n/4$ and $s_{n-1}(n) \le s_0(n)$,
we deduce $s_{n-1}(\alpha) \leq s_0(\alpha)$ for all $\alpha$, 
so we may discard $s_{n - 1}$.
\end{itemize}

In particular, when $n=2$ we get $s(\alpha) = s_0(\alpha)$. 
Thus, we continue with $n \geq 3$. 
\begin{itemize}
	\item If $m = n - 2$,  it suffices to show that 
$s_{n-2}(\alpha) \le s_0(\alpha)$ for $\alpha = n/2$ and $n-1/2$.
When $n \ge 4$ we have 
\begin{gather}
s_{n-2}(n/2) = \frac{n + 3}{8} \le \frac{n}{4} = s_0(n/2) \quad  \iff  \quad 
	3 \le n, \\
s_{n-2}(n - 1/2) = \frac{1}{2} \le \frac{3n - 1}{4(n + 1)} =  s_0(n - 1/2)  \quad   \iff  \quad
	3 \le n.
\end{gather}
When $n = 3$ the point $\alpha = n/2$ changes, but we still have
\begin{equation}
s_{n-2}(n/2) = \frac{11}{16} < \frac{3}{4} = s_0(n/2).
\end{equation}
Hence, we may discard $s_{n - 2}$. 
\end{itemize}

In particular, if $n=3$ we get $s(\alpha) = s_0(\alpha)$, 
and if $n=4$ we get $s(\alpha) = \max\{s_0(\alpha), s_1(\alpha)\}$. 
Thus, we continue with $n \ge 5$.

\begin{itemize}
	\item Let $m_1 < m \le n - 3$.
	From \eqref{eq:thm:m_large},
	it suffices to show that 
$s_m(\alpha) \le s_0(\alpha)$ for $\alpha \in \{ n/2,  m + 1, (n + m + 1)/2 \}$.
For $\alpha = n/2$, we have $s_m(n/2) = n/4 = s_0(n/2)$.
For $\alpha = m + 1$,
\begin{equation}
s_m(m+1) = \frac{n - m - 1}{2} \le 
	 \frac{n}{2(n + 1)} + \frac{n - 1}{2(n + 1)}(n - m - 1) = s_0(m+1)
\end{equation}
holds if and only if $n/2 - 1 \le m$.
In particular, it holds for $m > m_1$. 
For $\alpha = (n + m + 1)/2$, 
\begin{equation}
s_m(\alpha) = \frac{n - m}{4} \le 
	 \frac{n}{2(n + 1)} + \frac{n - 1}{4(n + 1)}(n - m -1) = s_0(\alpha)
\end{equation}
holds if and only if $m \geq (n-1)/2$. In particular, it holds when $m > m_1$. 
\end{itemize}
\end{proof}

Now we determine the maximum regularity among the small $m$.
\begin{lem}\label{thm:Lemma_Small_m}
Let $n \ge 4$ and $m_0 = \floor{(n - 1)/3}$, and define $s^S(\alpha) = 
	\max_{0 \le m \le m_0} s_m(\alpha)$. 
	Then,
\begin{equation}
s^S(\alpha) 
= \begin{cases} \label{eq:thm:sS}
s_{3,m_0}(\alpha), &
	\qquad n/2 \le \alpha \le n - m_0, \\
s_{4,m}(\alpha), & 
	\qquad  n - m \le \alpha \le n - m + \dfrac{n - 2m}{n - m}, \\
s_{3,m-1}(\alpha), &
	\qquad  n - m + \dfrac{n - 2m}{n - m} \le \alpha \le n - m + 1,
\end{cases}
\end{equation}
where $m$ ranges from 1 to $m_0$. 
Moreover, 
\begin{equation} \label{eq:thm:sS_maximum_dim_large}
s(\alpha) = s^S(\alpha), \qquad
\textrm{for } n - m_0 \le \alpha \le n.
\end{equation}
In particular, 
\begin{equation} \label{eq:thm:sS_s}
s(\alpha) = s^S(\alpha), \qquad
\textrm{for } \quad n/2 \le \alpha \le n, \qquad \text{ when } n=4, 5, 7. 
\end{equation}
\end{lem}

\begin{proof}
Let us prove \eqref{eq:thm:sS} with the aid of Figure~\ref{fig:Proto_Stairs}.
For $0 \le m \le m_0$ we have from \eqref{eq:thm:m_small}
that $s_m(n/2) = s_{3,m}(n/2) = n/4$, and also that
\begin{equation}
\textrm{slope of } s_{3,m}(\alpha) = -\frac{n - m - 1}{2(n - m + 1)},
\end{equation}
which is an increasing function of $m$, that is,
the smaller the $m$, the steeper the slope.
Hence,
\begin{equation} \label{eq:max_a_small}
 l \le m - 1 \quad \Longrightarrow \quad s_{m-1}(\alpha) =  s_{3,m - 1}(\alpha) \ge s_l(\alpha), 
 \qquad n/2 \le \alpha \le n - m + 1.
\end{equation}
On the other hand,
when $l \ge m$ and $n - m \le \alpha \le n$ we have
\begin{equation} \label{eq:max_a_big}
s_m(\alpha) = s_{4,m}(\alpha) \ge s_{4,l}(\alpha) =s_l(\alpha) 
\quad \iff \quad  
\frac{n - m}{n - m + 1} \ge \frac{n - l}{n - l + 1} 
\quad \iff \quad 
l \ge m.
\end{equation}
Together, \eqref{eq:max_a_small} and \eqref{eq:max_a_big} imply
\begin{equation}
s^S(\alpha) = \max\{s_{3,m-1}(\alpha), s_{4,m}(\alpha)\}, \qquad
 n - m \le \alpha \le n - m + 1.
\end{equation}
The last two cases in \eqref{eq:thm:sS} follow. 
The first case follows from \eqref{eq:max_a_small} with $m - 1 = m_0$.

\begin{figure}[h]
\centering
\includegraphics[width=0.64\textwidth]{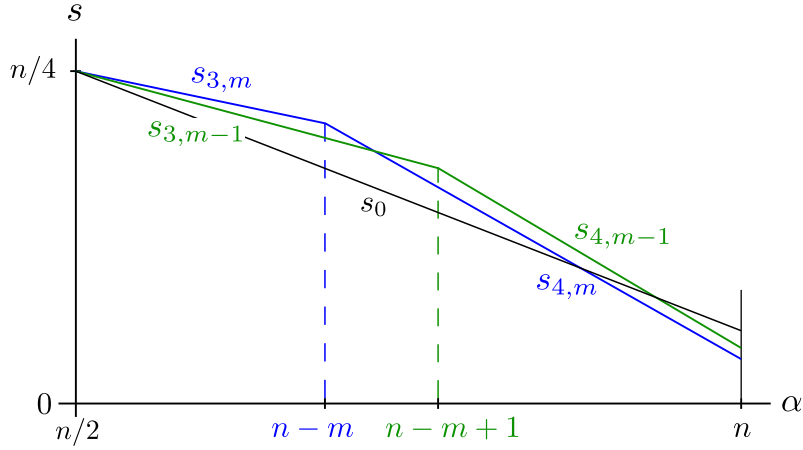}
\caption{Comparison between \textcolor{blue}{$s_m$} 
	and \textcolor{darkGreen}{$s_{m-1}$};
	see Proposition~\ref{thm:SobolevRegularity}\ref{it:m_small}. }
\label{fig:Proto_Stairs}
\end{figure}

To prove \eqref{eq:thm:sS_maximum_dim_large}, 
we need to discard the contribution of $m_0 < m \leq m_1$
in the range $n - m_0 \le \alpha \le n$.
Since $s_{m_0}(\alpha) = s_{4, m_0}(\alpha) \le s^S(\alpha)$ 
in this range of $\alpha$,
then we are done if we can show that $s_m(\alpha) \le s_{4, m_0}(\alpha)$,
where $s_m$ is given by \eqref{eq:thm:m_medium}. 

In the range $\beta_{m, 2} \le \alpha \le n$
we can repeat the analysis in \eqref{eq:max_a_big} to see that
\begin{equation} \label{eq:sS_large_large_dim}
	m_0 + 1 \le m \le m_1 \quad \Longrightarrow \quad s_m(\alpha) = 
		s_{4,m}(\alpha) < s_{4,m_0}(\alpha), \qquad
	 \beta_{m, 2} \le \alpha \le n.
\end{equation}
If $n \in 3\N$ and $m = m_0 + 1$, 
then $\beta_{m_0 + 1,\,2} < n - m_0$ and we are done;
otherwise, we have to consider the interval 
$n - m_0 \leq \alpha \leq \beta_{m,2}$ as well.

Assume that $n \notin 3\N$ or that $m > m_0 + 1$.
In this case, 
\begin{equation} \label{eq:end_point_large}
\beta_{m, 2} = \frac{n + m + 1}{2} \ge n - m_0.
\end{equation}
Since $s_m(\beta_{m, 2}) < s_{4, m_0}(\beta_{m, 2})$
and $s_m(\alpha) = s_{5,m}(\alpha)$ for $n - m_0 \leq \alpha \leq \beta_{m,2}$, 
it suffices to show that 
the slope of $s_{5,m}$ is greater (or less steep) than that of $s_{4,m_0}$, 
which is true because 
\begin{equation}
-\frac{n - m - 2}{2(n - m - 1)} \ge 
	-\frac{n - m_0}{2(n - m_0 + 1)} \iff
m \ge m_0 - 2.
\end{equation}
This concludes the proof of \eqref{eq:thm:sS_maximum_dim_large}.

Finally, \eqref{eq:thm:sS_s} holds because  for $n = 4, 5, 7$
there is no $m$ such that $m_0 < m \le m_1$.
\end{proof}

The next step is to determine 
the maximum regularity among the intermediate $m$.
\begin{lem}\label{thm:Lemma_Intermediate_m}
Let $n = 6$ or $n \geq 8$, 
$m_0 = \floor{(n - 1)/3}$ and $m_1 = \floor{n/2 - 1}$. 
Define $s^I(\alpha) = 
	\max_{m_0 < m \le m_1} s_m(\alpha)$.
Then,
\begin{equation} \label{eq:thm:medium_m_small_n}
s^I(\alpha)  = s_{m_0 + 1}(\alpha), \qquad
\textrm{when } n = 6, 8, 9, 10, 11, 13,
\end{equation}
and if $n=12$ or $n \geq 14$, 
\begin{equation} \label{eq:thm:medium_m}
s^I(\alpha) = \begin{cases}
s_{3, m_1}(\alpha) & 
	\textrm{for } n/2 \le \alpha \le \beta_{m_1, 1}
	\textrm{ and } n \textrm{ odd}, \\
s_{5, m}(\alpha), & 
	\textrm{for } \beta_{m, 1} \le \alpha \le n - m - 1 +3\,\dfrac{n - 2m - 2}{n - m - 4}, \\
s_{3, m - 1}(\alpha), &
	\textrm{for } n - m - 1 +3\,\dfrac{n - 2m - 2}{n - m - 4} \le \alpha \le \beta_{m - 1, 1}, \\
s_{m_0 + 1}(\alpha), &
	\textrm{for } \beta_{m_0 + 1, 1} \le \alpha \le n - m_0,
\end{cases}
\end{equation}
where $m$ ranges from $m_0 + 2$ to $m_1$.

Furthermore,
\begin{equation} \label{eq:thm:medium_dom_dim_small}
s(\alpha) = s^I(\alpha), \qquad
\textrm{for } n/2 \le \alpha \le \beta_{m_0 + 1, 1}.
\end{equation}
\end{lem}

\begin{proof}
The identity \eqref{eq:thm:medium_m_small_n} holds because
the interval $m_0 < m \le m_1$ only has one element for those dimensions.

To prove \eqref{eq:thm:medium_m} for $n=12$ or $n \geq 14$,
the analysis is like in Lemma~\ref{thm:Lemma_Small_m}.
Recall that $s_m$ is given by \eqref{eq:thm:m_medium} in this case,
so we have to consider the transition point
\begin{equation}
\beta_{m, 1} = n - m - 1 + 2\frac{n - 2m - 2}{n - m - 3} \in [n - m - 1, n - m).
\end{equation}
Like in \eqref{eq:max_a_small}, we see that
\begin{equation}\label{eq:max_a_intermedium}
l \leq m-1 \quad \Longrightarrow \quad s_{m-1}(\alpha) = s_{3,m-1}(\alpha) \geq s_{3,l}(\alpha) = s_l(\alpha), \qquad n/2 \leq \alpha \leq \beta_{m - 1,1}. 
\end{equation}
On the other hand, when $l\ge m$ and 
$\beta_{m, 1} \le \alpha \le \beta_{m_0 + 1, 1} < \beta_{m, 2}$ we have
\begin{equation} \label{eq:s_5_maximum}
s_m(\alpha) = s_{5,m}(\alpha) \ge s_{5,l}(\alpha) =  s_l(\alpha) \iff
 \frac{n - m - 2}{n - m - 1} \geq \frac{n - l - 2}{n - l - 1} \iff
l \ge m.
\end{equation}
Consequently, 
\begin{equation}
s^I(\alpha) = \max\{s_{3, m-1}(\alpha), s_{5, m}(\alpha)\}, \qquad
 \beta_{m, 1} \le \alpha \le \beta_{m - 1, 1},
\end{equation}
and the two middle cases in \eqref{eq:thm:medium_m} follow.
 
When $n$ is even we get $\beta_{m_1, 1} = n/2$, 
so the computations above cover the whole range $n/2 \le \alpha \le \beta_{m_0 + 1, 1}$. 
When $n$ is odd, though, 
$\beta_{m_1, 1} > n/2$ and the first case in \eqref{eq:thm:medium_m} 
follows from \eqref{eq:max_a_intermedium} by taking $m - 1 = m_1$. 

Now we prove the last case in \eqref{eq:thm:medium_m}, that is,
that $s_{m_0 + 1}(\alpha) \ge s_m(\alpha)$ for $m_0 + 2 \le m \le m_1$
and for $\beta_{m_0 + 1, 1} \le \alpha \le n - m_0$.
From \eqref{eq:thm:m_medium} and \eqref{eq:end_point_large}
we see that $s_m(\alpha) = s_{5, m}(\alpha)$, so
we have to prove $s_{m_0 + 1}(\alpha) \ge s_{5, m}(\alpha)$
for $m \ge m_0 + 2$.

When $n \not\in 3\N$ we have $\beta_{m_0 + 1,\,2} \ge n - m_0$,
so $s_{m_0 + 1}(\alpha) = s_{5, m_0 + 1}(\alpha)$ and
we have to prove $s_{5, m_0 + 1}(\alpha)
\ge s_{5, m}(\alpha)$, but this follows 
like in \eqref{eq:s_5_maximum}.
When $n \in 3\N$, then $m_0 = n/3 - 1$ and
we also have to study the range
\begin{equation}
\beta_{m_0 + 1,\,2} = \frac{n + (m_0 + 1) + 1}{2} = \frac{2n}{3} + \frac{1}{2} \le \alpha \le  \frac{2n}{3} + 1 = n - m_0.
\end{equation}
In this range $s_{m_0 + 1}(\alpha) = s_{4, m_0 + 1}(\alpha)$, so
we must prove $s_{4, m_0 + 1}(\alpha) \ge s_{5, m}(\alpha)$ for $m \ge m_0 + 2$.
For that, it is enough to check $s_{4, m_0 + 1}(\alpha) \ge s_{5, m_0 + 2}(\alpha)$.
Since $s_{m_0 + 1}(\beta_{m_0 + 1,\, 2}) \ge s_{5, m_0 + 2}(\beta_{m_0 + 1,\, 2})$,
we only need to prove the inequality at 
$\alpha = 2n/3 + 1$, 
which follows after algebraic manipulation.

To prove \eqref{eq:thm:medium_dom_dim_small},
by the first case in \eqref{eq:thm:sS}
it is enough to show that 
$s_{3, m_0}(\alpha) < s_{3, m_0 + 1}(\alpha) = s_{m_0 + 1}(\alpha) \le s^I(\alpha)$ 
for $n/2 \le \alpha \le \beta_{m_0 + 1, 1}$.
This follows like in \eqref{eq:max_a_small}, so
we conclude the proof of the lemma.
\end{proof}

After Lemmas~\ref{thm:Lemma_Small_m} and \ref{thm:Lemma_Intermediate_m}, 
it only remains to analyze the range $\beta_{m_0 + 1, 1} \le \alpha \le n - m_0$.
\begin{lem}
Let $m_0 = \floor{(n - 1)/3}$. Then,
\begin{itemize}
	\item When $n = 6$,
\begin{equation} \label{eq:thm:transtion_6}
s(\alpha) = s_{3, m_0}(\alpha), \qquad
\textrm{for } n/2 \le \alpha \le n - m_0.
\end{equation}
\item When $n \ge 8$, 
\begin{equation} \label{eq:thm:transtion_n3}
s(\alpha) = \begin{cases}
s_{5, m_0 + 1}(\alpha), & 
	\qquad \beta_{m_0 + 1, 1} \le \alpha \le n - m_0 - 2 + 3\,\dfrac{n - 2m_0 - 4}{n - m_0 - 5}, \\
s_{3, m_0}(\alpha), & 
	\qquad n - m_0 - 2 + 3\,\dfrac{n - 2m_0 - 4}{n - m_0 - 5} \le \alpha \le n - m_0.
\end{cases}
\end{equation}
\end{itemize}
\end{lem}

\begin{proof}
From the first case in \eqref{eq:thm:sS} and 
the last case in \eqref{eq:thm:medium_m} we have that
\begin{equation}
s(\alpha) = \max\{s_{3, m_0}(\alpha), s_{m_0 + 1}(\alpha)\}.
\end{equation} 
When $n \notin 3\N$ then \eqref{eq:thm:m_medium} and 
\eqref{eq:end_point_large} imply
that $s_{m_0 + 1}(\alpha) = s_{5, m_0 + 1}(\alpha)$, so
\begin{equation}
s(\alpha) = \max\{s_{3, m_0}(\alpha), s_{5, m_0 + 1}(\alpha)\},
\end{equation}
which is precisely \eqref{eq:thm:transtion_n3}; 
notice that $n - m_0 - 5 > 0$ in this case.
When $n \in 3\N$ and $n \neq 6$, then
$m_0 = n/3 - 1$ and 
\begin{equation}
s(\alpha) = \begin{cases}
\max\{s_{3, m_0}(\alpha), s_{5, m_0 + 1}(\alpha)\}, & 
	\qquad \beta_{m_0 + 1, 1} \le \alpha \le 2n/3 + 1/2, \\
\max\{s_{3, m_0}(\alpha), s_{4, m_0 + 1}(\alpha)\}, &
	\qquad 2n/3 + 1/2 \le \alpha \le 2n/3 + 1,
\end{cases}
\end{equation}
which again leads to \eqref{eq:thm:transtion_n3}; 
notice that
\begin{equation}
	n - m_0 - 2 + 3\,\frac{n - 2m_0 - 4}{n - m_0 - 5} = 
	\beta_{m_0 + 1,\, 2} = \frac{2n}{3} + \frac{1}{2}.
\end{equation}
When $n = 6$, we have $s_{3, m_0}(\alpha) = s_{5, m_0 + 1}(\alpha)$, so
we may choose $s_{3, m_0}$ in the first maximum above to reach \eqref{eq:thm:transtion_6}.
\end{proof}

Gathering the results of this section, we get Theorem~\ref{thm:Main_Theorem}.
\begin{rmk}
Given that $\beta_{m,2}$ plays no role in the final statement of Theorem~\ref{thm:Main_Theorem},
for simplicity we rename $\beta_{m,1}$ as $\beta_m$.
\end{rmk}

\bibliographystyle{acm}
\bibliography{SchrodingerBib}

\begin{thebibliography}{10}

\bibitem{Pierce2021}
{\sc An, C., Chu, R., and Pierce, L.~B.}
\newblock Counterexamples for high-degree generalizations of the
  {S}chr\"odinger maximal operator.
\newblock https://arxiv.org/abs/2103.15003.

\bibitem{bercelo_etal2011}
{\sc Barcel\'{o}, J.~A., Bennett, J., Carbery, A., and Rogers, K.~M.}
\newblock On the dimension of divergence sets of dispersive equations.
\newblock {\em Math. Ann. 349}, 3 (2011), 599--622.

\bibitem{BeresnevichVelani2006}
{\sc Beresnevich, V., and Velani, S.}
\newblock A mass transference principle and the {D}uffin-{S}chaeffer conjecture
  for {H}ausdorff measures.
\newblock {\em Ann. of Math. (2) 164}, 3 (2006), 971--992.

\bibitem{bourgain2013}
{\sc Bourgain, J.}
\newblock On the {S}chr\"{o}dinger maximal function in higher dimension.
\newblock {\em Tr. Mat. Inst. Steklova 280\/} (2013), 53--66.

\bibitem{Bourgain2016}
{\sc Bourgain, J.}
\newblock A note on the {S}chr\"{o}dinger maximal function.
\newblock {\em J. Anal. Math. 130\/} (2016), 393--396.

\bibitem{carleson1980}
{\sc Carleson, L.}
\newblock Some analytic problems related to statistical mechanics.
\newblock In {\em Euclidean harmonic analysis ({P}roc. {S}em., {U}niv.
  {M}aryland, {C}ollege {P}ark, {M}d., 1979)\/} (1980), vol.~779 of {\em
  Lecture Notes in Math.}, Springer, Berlin, pp.~5--45.

\bibitem{choLee2014}
{\sc Cho, C.-H., and Lee, S.}
\newblock Dimension of divergence sets for the pointwise convergence of the
  {S}chr\"{o}dinger equation.
\newblock {\em J. Math. Anal. Appl. 411}, 1 (2014), 254--260.

\bibitem{choLeeVargas2012}
{\sc Cho, C.-H., Lee, S., and Vargas, A.}
\newblock Problems on pointwise convergence of solutions to the
  {S}chr\"{o}dinger equation.
\newblock {\em J. Fourier Anal. Appl. 18}, 5 (2012), 972--994.

\bibitem{compaan_etal2021}
{\sc Compaan, E., Luc\`a, R., and Staffilani, G.}
\newblock Pointwise convergence of the {S}chr\"{o}dinger flow.
\newblock {\em Int. Math. Res. Not. IMRN}, 1 (2021), 599--650.

\bibitem{dahlberg_etal1982}
{\sc Dahlberg, B. E.~J., and Kenig, C.~E.}
\newblock A note on the almost everywhere behavior of solutions to the
  {S}chr\"{o}dinger equation.
\newblock In {\em Harmonic analysis ({M}inneapolis, {M}inn., 1981)}, vol.~908
  of {\em Lecture Notes in Math.} Springer, Berlin-New York, 1982,
  pp.~205--209.

\bibitem{Du2020}
{\sc Du, X.}
\newblock Upper bounds for {F}ourier decay rates of fractal measures.
\newblock {\em J. Lond. Math. Soc. (2) 102}, 3 (2020), 1318--1336.

\bibitem{du_etal2017}
{\sc Du, X., Guth, L., and Li, X.}
\newblock A sharp {S}chr\"{o}dinger maximal estimate in {$\Bbb R^2$}.
\newblock {\em Ann. of Math. (2) 186}, 2 (2017), 607--640.

\bibitem{DKWZ20}
{\sc Du, X., Kim, J., Wang, H., and Zhang, R.}
\newblock Lower bounds for estimates of the {S}chr\"{o}dinger maximal function.
\newblock {\em Math. Res. Lett. 27}, 3 (2020), 687--692.

\bibitem{du_etal2018}
{\sc Du, X., and Zhang, R.}
\newblock Sharp {$L^2$} estimates of the {S}chr\"{o}dinger maximal function in
  higher dimensions.
\newblock {\em Ann. of Math. (2) 189}, 3 (2019), 837--861.

\bibitem{eceizabarrena2020}
{\sc Eceizabarrena, D., and Luc\`a, R.}
\newblock Convergence over fractals for the periodic {S}chr\"{o}dinger
  equation.
\newblock {\em Anal. PDE, To Appear\/}.
\newblock https://arxiv.org/abs/2005.07581.

\bibitem{eceizabarrena2021}
{\sc Eceizabarrena, D., and Ponce-Vanegas, F.}
\newblock Pointwise convergence over fractals for dispersive equations with
  homogeneous symbol.
\newblock https://arxiv.org/abs/2108.10339.

\bibitem{falconer_book03}
{\sc Falconer, K.}
\newblock {\em Fractal geometry}, second~ed.
\newblock John Wiley \& Sons, Inc., Hoboken, NJ, 2003.
\newblock Mathematical foundations and applications.

\bibitem{hardyWright2008}
{\sc {Hardy}, G.~H., and {Wright}, E.~M.}
\newblock {\em {An introduction to the theory of numbers. Edited and revised by
  D. R. Heath-Brown and J. H. Silverman. With a foreword by Andrew Wiles. 6th
  ed}}, 6th ed.~ed.
\newblock Oxford: Oxford University Press, 2008.

\bibitem{lee2006}
{\sc Lee, S.}
\newblock On pointwise convergence of the solutions to {S}chr\"{o}dinger
  equations in {$\Bbb R^2$}.
\newblock {\em Int. Math. Res. Not.\/} (2006), Art. ID 32597, 21.

\bibitem{LiWangYan2021}
{\sc Li, W., Wang, H., and Yan, D.}
\newblock A note on non-tangential convergence for {S}chr\"{o}dinger operators.
\newblock {\em J. Fourier Anal. Appl. 27}, 4 (2021), Paper No. 61, 14.

\bibitem{LiZhao2021}
{\sc Li, Z., Zhao, J., and Zhao, T.}
\newblock {$L^2$} schr\"{o}dinger maximal estimates associated with finite type
  phases in $\mathbb{R}^2$.
\newblock https://arxiv.org/abs/2111.00897.

\bibitem{LucaPonceVanegas2021}
{\sc Lucà, R., and Ponce-Vanegas, F.}
\newblock Convergence over fractals for the {S}chr\"odinger equation.
\newblock {\em Indiana Univ. Math. J., To Appear\/}.
\newblock https://arxiv.org/abs/2101.02495.

\bibitem{lucaRogers2019}
{\sc Luc\`a, R., and Rogers, K.~M.}
\newblock Average decay of the {F}ourier transform of measures with
  applications.
\newblock {\em J. Eur. Math. Soc. (JEMS) 21}, 2 (2019), 465--506.

\bibitem{luca_etal2019}
{\sc Luc\`a, R., and Rogers, K.~M.}
\newblock A note on pointwise convergence for the {S}chr\"{o}dinger equation.
\newblock {\em Math. Proc. Cambridge Philos. Soc. 166}, 2 (2019), 209--218.

\bibitem{moyua1996}
{\sc Moyua, A., Vargas, A., and Vega, L.}
\newblock Schr\"{o}dinger maximal function and restriction properties of the
  {F}ourier transform.
\newblock {\em Internat. Math. Res. Notices}, 16 (1996), 793--815.

\bibitem{pierce2020}
{\sc Pierce, L.~B.}
\newblock On {B}ourgain's counterexample for the {S}chr\"{o}dinger maximal
  function.
\newblock {\em Q. J. Math. 71}, 4 (2020), 1309--1344.

\bibitem{sjogren_etal1989}
{\sc Sj\"{o}gren, P., and Sj\"{o}lin, P.}
\newblock Convergence properties for the time-dependent {S}chr\"{o}dinger
  equation.
\newblock {\em Ann. Acad. Sci. Fenn. Ser. A I Math. 14}, 1 (1989), 13--25.

\bibitem{sjogren1989}
{\sc Sj\"{o}gren, P., and Sj\"{o}lin, P.}
\newblock Convergence properties for the time-dependent {S}chr\"{o}dinger
  equation.
\newblock {\em Ann. Acad. Sci. Fenn. Ser. A I Math. 14}, 1 (1989), 13--25.

\bibitem{sjolin1987}
{\sc Sj\"{o}lin, P.}
\newblock Regularity of solutions to the {S}chr\"{o}dinger equation.
\newblock {\em Duke Math. J. 55}, 3 (1987), 699--715.

\bibitem{vega1988}
{\sc Vega, L.}
\newblock Schr\"{o}dinger equations: pointwise convergence to the initial data.
\newblock {\em Proc. Amer. Math. Soc. 102}, 4 (1988), 874--878.

\bibitem{zubrinic2002}
{\sc \v{Z}ubrini\'{c}, D.}
\newblock Singular sets of {S}obolev functions.
\newblock {\em C. R. Math. Acad. Sci. Paris 334}, 7 (2002), 539--544.

\bibitem{WangWu2021}
{\sc Wang, B., and Wu, J.}
\newblock {M}ass transference principle from rectangles to rectangles in
  {D}iophantine approximation.
\newblock {\em Math. Ann.\/} (2021).

\bibitem{wangZhang2019}
{\sc Wang, X., and Zhang, C.}
\newblock Pointwise convergence of solutions to the {S}chr\"{o}dinger equation
  on manifolds.
\newblock {\em Canad. J. Math. 71}, 4 (2019), 983--995.

\end{thebibliography}

\end{document}